\pdfoutput=1
\documentclass[11pt,a4paper,reqno]{amsart}
\usepackage[foot]{amsaddr}
\usepackage[utf8]{inputenc}
\usepackage[left=1in,right=1in,top=1in,bottom=1in]{geometry}

\usepackage{bbm}
\usepackage{amssymb}
\usepackage{amsmath}
\usepackage{mathtools}
\usepackage{mathdots}
\usepackage{array,multirow,graphicx}
\usepackage{quiver}

\usepackage{amsthm}
\usepackage{enumitem}
\usepackage[hidelinks,unicode,naturalnames,hypertexnames=false]{hyperref}
\usepackage{thmtools}
\usepackage[english,noabbrev,sort,nameinlink]{cleveref}


\newcounter{mycounter}[section]

\theoremstyle{plain}
\newtheorem{theorem}[mycounter]{Theorem}
\newtheorem{proposition}[mycounter]{Proposition}
\newtheorem{lemma}[mycounter]{Lemma}
\newtheorem{corollary}[mycounter]{Corollary}

\theoremstyle{definition}
\newtheorem{definition}[mycounter]{Definition}

\theoremstyle{remark}
\newtheorem{remark}[mycounter]{Remark}

\numberwithin{equation}{section}
\crefname{subsection}{subsection}{subsections}

\setlist[itemize]{noitemsep, topsep=0pt, leftmargin=1cm}

\newlist{facts}{enumerate}{1}
\setlist[facts, 1]{label=\normalfont\bfseries(\Roman*)}
\newlist{enumalpha}{enumerate}{1}
\setlist[enumalpha, 1]{label=(\alph*)}
\newlist{enumroman}{enumerate}{1}
\setlist[enumroman, 1]{label=(\roman*)}

\newcommand{\leftl}{\mathopen{}\mathclose\bgroup\left}
\newcommand{\rightr}{\aftergroup\egroup\right}
\renewcommand{\bar}{\overline}
\renewcommand{\tilde}{\widetilde}
\renewcommand{\hat}{\widehat}

\newcommand{\bigslant}[2]{{\raisebox{.2em}{$#1$}\left/\raisebox{-.2em}{$#2$}\right.}}
\newcommand{\prim}{{}\!\!\!{\textnormal{\LARGE\textquoteright}}}

\newcommand{\Z}{\mathbb Z}
\newcommand{\Q}{\mathbb Q}
\newcommand{\N}{\mathbb N}
\newcommand{\F}{\mathbb F}
\newcommand{\R}{\mathbb R}
\newcommand{\cD}{\mathcal D}
\newcommand{\cI}{\mathcal I}
\newcommand{\cO}{\mathcal O}
\newcommand{\cP}{\mathcal P}
\newcommand{\fp}{\mathfrak p}
\newcommand{\fq}{\mathfrak q}

\newcommand{\perf}{\textnormal{perf}}
\newcommand{\ab}{\textnormal{ab}}
\newcommand{\ur}{\textnormal{ur}}
\newcommand{\sep}{\textnormal{sep}}
\newcommand{\alg}{\textnormal{alg}}

\DeclareMathOperator{\Gal}{Gal}
\DeclareMathOperator{\Hom}{Hom}

\DeclareMathOperator{\height}{ht}
\DeclareMathOperator{\lastjump}{lastjump}
\DeclareMathOperator{\cond}{cond}
\DeclareMathOperator{\minlift}{minlift}

\newcommand{\notp}{\N\setminus p\N}
\newcommand{\Onotp}{\{0\}\cup\notp}
\newcommand{\nottwo}{\N\setminus 2\N}
\newcommand{\Onottwo}{\{0\}\cup\nottwo}
\newcommand{\card}[1]{\leftl\lvert#1\rightr\rvert}
\newcommand{\suchthat}[2]{\left\{\begin{matrix}#1\end{matrix}\ \middle\vert\ \begin{matrix}#2\end{matrix}\right\}}
\newcommand{\cardsuchthat}[2]{\card{\suchthat{#1}{#2}}}
\newcommand{\simto}{\overset\sim\to}
\newcommand{\ppar}[1]{(\!(#1)\!)}
\newcommand{\mat}[1]{\left(\begin{smallmatrix}#1\end{smallmatrix}\right)}
\renewcommand{\d}{\mathrm{d}}

\newcommand{\ceil}[1]{\leftl\lceil#1\rightr\rceil}

\newcommand{\restfp}{|_{\Gamma_{K_\fp}}}

\newcommand{\customref}[2]{\hyperref[#2]{#1~\ref*{#2}}}
\newcommand{\iref}[2]{\customref{\Cref*{#1}}{#2}}

\title{Lifts of unramified twists and local--global principles}

\author{Fabian Gundlach}
\email{fabian.gundlach@uni-paderborn.de}
\author{Béranger Seguin}
\email{math@beranger-seguin.fr}
\address{Universität Paderborn, Fakultät EIM, Institut für Mathematik, Warburger Str.~100, 33098 Paderborn, Germany.}

\begin{document}

\begin{abstract}
  We prove that two-step nilpotent $p$-extensions of rational global function fields of characteristic~$p$ satisfy a quantitative local--global principle when they are counted according to their largest upper ramification break (``last jump'').
  We had previously shown this only for $p\neq2$.
  Compared to our previous proof, this proof is also more self-contained, and may apply to heights other than the last jump.
  As an application, we describe the distribution of last jumps of $D_4$-extensions of rational global function fields of characteristic~$2$.
  We also exhibit a counterexample to the analogous local--global principle when counting by discriminants.
\end{abstract}

\maketitle

{
  \setcounter{tocdepth}{1}
  \hypersetup{linkcolor=black}
  \tableofcontents{}
}

\section{Introduction}
\label{sn:introduction}

For the whole article, we fix a prime number~$p$.
If~$K$ is a field, $\Gamma_K := \Gal(K^\sep|K)$ is its absolute Galois group equipped with the Krull topology.
For a finite group~$G$ with the discrete topology, $\Hom(\Gamma_K, G)$ is the set of continuous homomorphisms $\Gamma_K \to G$.
We denote the higher inertia groups of a local field~$K$ by~$\Gamma_K^v$ (in the upper numbering, for any $v \in \R_{\geq-1}$) and, for any $\rho\in\Hom(\Gamma_K, G)$, we define
\[
  \lastjump\rho :=
  \inf\suchthat{
      v \in \Q_{\geq 0}
  }{
      \rho(\Gamma_K^v) = 1
  }
  \in \Q_{\geq0}.
\]
Finally, if~$K$ is a global field, the set of its places is denoted by~$\cP_K$, and for each $\fp \in \cP_K$, we let~$K_\fp$ be the completion of~$K$ at~$\fp$ and we let $\cO_\fp := \cO_{K_\fp}$.

\medskip

One of the main themes of class field theory is the interplay between the local theory and the global theory, leading to a local--global principle for abelian extensions (cf.~\Cref{lem:cft}).
In this article, we prove an analogous quantitative local--global principle for certain non-abelian extensions when they are ordered by their last jump:

\begin{theorem}
  \label{thm:intro-local-global}
  Let~$G$ be a finite $p$-group of nilpotency class~$\leq 2$ and let $K := \F_q(T)$ be a rational global function field of characteristic~$p$.
  Then, for any $(v_\fp) \in \prod_{\fp \in \cP_K} \Q_{\geq0}$ such that~$v_\fp = 0$ for almost all~$\fp \in \cP_K$, we have
	\begin{align*}
		&\frac{1}{\card G}
    \cardsuchthat{
			\rho \in \Hom(\Gamma_K, G)
    }{
			\forall \fp,\ \lastjump \bigl( \rho\restfp \bigr) = v_\fp
		} \\
		={}&
		\prod_{\fp \in \cP_K}
  		\frac{1}{\card G}
      \cardsuchthat{
        \rho_\fp \in \Hom(\Gamma_{K_\fp}, G)
      }{
			  \lastjump \rho_\fp = v_\fp
      },
	\end{align*}
  where the left-hand side and all the factors of the right-hand side are nonnegative integers.
  (For all primes $\fp$ with $v_\fp = 0$,
  the factor is~$1$, making the infinite product well-defined.)
\end{theorem}

For non-abelian groups~$G$, we are not aware of a bijection directly underlying this quantitative local-global principle.
(For abelian~$G$, see \Cref{lem:cft}.)

This local-global principle is not a general fact about inertial heights; for instance, it can fail when one replaces the last jump by the degree of the discriminant divisor as we observe in \Cref{prop:counterexample-discr}.
However, we show in \Cref{thm:single-prime} that it does hold (for any finite $p$-group and any inertial height) for extensions unramified outside a single place of degree not divisible by~$p$.
Moreover, it is likely that in general situations an approximate statistical version of the quantitative local--global principle holds when averaging over many tuples~$(v_\fp)$.

A consequence of \Cref{thm:intro-local-global} is that, when~$G$ is a finite $p$-group of nilpotency class~$\leq 2$, the problem of describing the asymptotic distribution of last jumps $\sum_\fp \deg\fp\cdot\lastjump(\rho|_{\Gamma_{K_\fp}})$ of $G$-extensions of the global field~$\F_q(T)$ reduces to the study of the distribution of last jumps of $G$-extensions of the local function fields~$\F_q\ppar{T}$ using an analytic lemma~\cite[Lemma~5.4]{wildcount}.
As an application, we count $D_4$-extensions of~$\F_q(T)$ in characteristic~$2$ (a case that we could not treat in \cite{wildcount}) using the results of \cite{imai-wild-ramification-groups}:

\begin{theorem}[\Cref{thm:count-global-D4}]
  Let $K := \F_q(T)$ be a rational global function field of characteristic~$2$.
  There is a real number~$C > 0$ such that for integers~$X\to\infty$, we have
  \[
    \cardsuchthat{
      \rho \in \Hom(\Gamma_K, D_4)
    }{
      \sum_{\fp \in \cP_K}
        \deg \fp \cdot \lastjump\leftl(\rho\restfp\rightr)
      =
      X
    }
    =
    C q^{3 X} X  + o(q^{3X} X).
  \]
\end{theorem}

\subsection{Relation with previous results.}

When~$p = 2$ and~$G$ has nilpotency class~$2$, \Cref{thm:intro-local-global} is a new result.

When~$G$ is abelian (in which case the last jump is closely related to the conductor), \Cref{thm:intro-local-global} is a straightforward consequence of class field theory (see \Cref{lem:cft}), and in this case it is not specific to the last jump (it holds for any inertial height, e.g., for discriminants).

When~$p \geq 3$, a proof of \Cref{thm:intro-local-global} was already given in \cite[Theorem~1.1]{wildcount}, relying on the Lazard correspondence between $p$-groups of nilpotency class less than~$p$ (leading to the exclusion of the prime~$2$) and finite Lie $\Z_p$-algebras, and on the description of the ramification filtration from~\cite{abrashkin-ramification-filtration-3}.
By contrast, this article deals directly with continuous homomorphisms~$\Gamma_K \to G$ and uses only classical results from class field theory and Artin--Schreier--Witt theory.

\subsection{Strategy of proof}

The assumption that~$G$ is a finite $p$-group of nilpotency class~$\leq2$ is equivalent to the existence of a short exact sequence
\[
  1 \longrightarrow N \longrightarrow G \stackrel\pi\longrightarrow Q \longrightarrow 1
\]
for two finite abelian $p$-groups $N \subseteq Z(G)$ and~$Q$.
Any continuous homomorphism $\rho \colon \Gamma_K \to G$ gives rise to a continuous homomorphism $\bar\rho = \pi\circ\rho \in \Hom(\Gamma_K, Q)$.
Conversely, in this case,
every $\bar\rho \in \Hom(\Gamma_K, Q)$ admits a \emph{lift} $\rho \in \Hom(\Gamma_K, G)$,
and the other lifts of~$\bar\rho$ are exactly the \emph{twists}~$\delta\cdot\rho$ (pointwise products) of~$\rho$ by homomorphisms $\delta \in \Hom(\Gamma_K, N)$ (see \Cref{lem:embpb-solvable}):
\[\begin{tikzcd}
  & & \Gamma_K\ar{dr}{\bar\rho}\ar[swap,shift right=3pt]{d}{\rho}\ar{d}\ar[shift left=3pt]{d}\ar[black!40,swap,shift right=3pt]{dl}{\delta}\ar[black!40]{dl}\ar[black!40,shift left=3pt]{dl} \\
  1\rar & N\rar & G\rar[swap]{\pi} & Q\rar & 1
\end{tikzcd}\]

Since~$Q$ and~$N$ are abelian, $\Hom(\Gamma_K, Q)$ and~$\Hom(\Gamma_K, N)$ can be understood using class field theory.
In fact, the assumption that~$K$ is a rational function field of characteristic~$p$ implies a local--global principle for homomorphisms modulo unramified homomorphisms (\Cref{lem:cft}):
\begin{equation}
  \label{eq:locglob-for-Q}
  \bigslant
    {\Hom(\Gamma_K, Q)}
    {\Hom_\ur(\Gamma_K, Q)}
  \simeq
  \bigoplus_{\fp\in\cP_K}
  \Bigl(
    \bigslant
      {\Hom(\Gamma_{K_\fp}, Q)}
      {\Hom_\ur(\Gamma_{K_\fp}, Q)}
  \Bigr).
\end{equation}
The quotient groups $\Hom(\Gamma_{K_\fp}, Q) / \Hom_\ur(\Gamma_{K_\fp}, Q)$ can be thought of as follows: two homomorphisms $\Gamma_{K_\fp}\to Q$ lie in the same coset (they are \emph{unramified twists} of each other) if and only if they agree on the inertia subgroup $\Gamma_{K_\fp}^0 \subset \Gamma_{K_\fp}$.

In order to prove \Cref{thm:intro-local-global}, our basic strategy is to combine the local--global principle for~$\Hom(\Gamma_K, Q)$ with that for $\Hom(\Gamma_K, N)$, following a classical inductive approach.
(See for example \cite{aloww-inductive-methods} for a recent treatment of inductive methods.)
However, the fact that the local--global principle for~$\Hom(\Gamma_K, Q)$ holds only \emph{up to unramified homomorphisms} leads to a difficulty in implementing this idea naively: how can we count the lifts of a homomorphism that we know only up to unramified homomorphisms?
The simplest way to get around this indeterminacy is to prove that the distribution of the last jumps of the lifts $\rho \in \Hom(\Gamma_{K_\fp}, G)$ of a given $\bar\rho \in \Hom(\Gamma_{K_\fp}, Q)$ is in fact unchanged when replacing~$\bar\rho$ with an unramified twist.
This is the content of the following purely local theorem, which is the main technical ingredient of this paper:

\begin{theorem}
  \label{thm:minembed-inertial}
  Let~$K := \F_q\ppar{T}$ be a local function field of characteristic~$p$.
  Consider a short exact sequence $1 \to N \to G \stackrel\pi\to Q \to 1$ with finite abelian $p$-groups $N \subseteq Z(G)$ and~$Q$.
  For any~$\bar\rho_1, \bar\rho_2 \in \Hom(\Gamma_K, Q)$ such that~$\bar\delta := \bar\rho_2 - \bar\rho_1$ is unramified and for any $v \geq 0$, we have
  \[
    \cardsuchthat{\rho \in \Hom(\Gamma_K, G)}{
      \pi\circ\rho = \bar\rho_1 \\
      \lastjump \rho = v
    }
    =
    \cardsuchthat{\rho \in \Hom(\Gamma_K, G)}{
      \pi\circ\rho = \bar\rho_2 \\
      \lastjump \rho = v
    }.
  \]
\end{theorem}

\begin{center}
\begin{tikzpicture}[y=1.3cm]
  \draw[black!20] (-3,0) -- (8,0);
  \draw[black!20] (-3,1.5) -- (8,1.5);
  \draw[black!20] (-3,2) -- (8,2);
  \draw[black!20] (-3,2.3) -- (8,2.3);
  \draw[->] (-3,-0.3) -- node[above, sloped, overlay] {$\lastjump$} (-3,3);
  \newcommand{\stuff}[4]{
    \node[circle, inner sep=0.2em, fill=white] (barrho#1) at (0,0) {$\bar\rho_#1$};
    \path[#2] (-1.5,1.5) node[fill=white] {$\rho$} edge node[below,sloped] {#3} (barrho#1);
    \path (-0.65,1.5) node[fill=white] {$\rho$} edge (barrho#1);
    \path (0,2) node[fill=white] {$\rho$} edge (barrho#1);
    \path (0.75,2.3) node[fill=white] {$\rho$} edge (barrho#1);
    \node at (1.3,2.7) {$\iddots$};
    \draw[->, rounded corners, black!50] (-1.6,1.75) -- (-0.65,1.75) -- (0,2.25) node[above, rotate=35] {#4} --  (0.75,2.55) -- (1.3,2.95);
  }
  \stuff{1}{|->}{$\pi\circ$}{\tiny twist by $\delta\colon\Gamma_K\to N$}
  \begin{scope}[shift={(5,0)}]
    \stuff{2}{}{}{}
  \end{scope}
  \draw[<->, decorate, decoration={zigzag, segment length=4, amplitude=0.8, post length=3pt, pre length=3pt}] (barrho1) -- node[below] {unramified twists} (barrho2);
\end{tikzpicture}
\end{center}

To prove \Cref{thm:minembed-inertial}, it in fact suffices to prove that the smallest last jump among all lifts of~$\bar\rho_1$ equals the smallest last jump among all lifts of~$\bar\rho_2$.
We achieve this by explicitly constructing for every lift~$\rho_1$ of~$\bar\rho_1$ and for every unramified lift~$\delta$ of~$\bar\delta := \bar\rho_2-\bar\rho_1$ a lift $\rho_2 = \varepsilon\cdot\delta\cdot\rho_1$ of $\bar\rho_2 = \bar\delta\cdot\bar\rho_1$ with $\lastjump\rho_2 \leq \lastjump\rho_1$.
The factor $\varepsilon \colon \Gamma_{K_\fp} \to N$ is defined by an explicit formula, relying on Artin--Schreier--Witt theory.
(See \Cref{subsn:proof-minembed-inertial}.)

Informally, our construction extends unramified twisting to the non-abelian setting, connecting the two sides of \Cref{thm:minembed-inertial}.
This modifies not only the unramified part ($\rho_1$ and~$\rho_2$ need not coincide when restricted to the inertia subgroup), but also affects the first few ramification subgroups (up to~$\lastjump \bar\rho_1$).
This partially explains the specificity of the last jump: this height is insensitive to modifying only the first few ramification subgroups, as opposed to the discriminant which depends on the sizes of all ramification subgroups (\Cref{eq:def-discr-exp}).

\subsection{Organization of the paper}

In \Cref{sn:preliminaries}, we recall facts about $p$-extensions, class field theory and Artin--Schreier--Witt theory; we also define the notions of \emph{unramified twists}, \emph{inertial types}, \emph{inertial heights} and the \emph{last jump}.
In \Cref{sn:locglob}, we prove \Cref{thm:abstract-locglob}, which states that any inertial height satisfying a form of \Cref{thm:minembed-inertial} satisfies a local--global principle similar to \Cref{thm:intro-local-global}.
\Cref{sn:minembed-urtwist} is dedicated to proving \Cref{thm:minembed-inertial}, which together with \Cref{thm:abstract-locglob} implies \Cref{thm:intro-local-global}.
\Cref{sn:counting-d4} illustrates the use of \Cref{thm:intro-local-global} by applying it, together with \cite[Lemma~5.4]{wildcount} and with the results of \cite{imai-wild-ramification-groups}, to describe the asymptotics of $D_4$-extensions of rational global function fields in characteristic~$2$ (\Cref{thm:count-global-D4}).
In \Cref{sn:generalizations}, we discuss generalizations of \Cref{thm:intro-local-global}: we describe a more general local--global principle for extensions unramified outside a single prime of degree coprime to~$p$ (\Cref{thm:single-prime}), and we give a counterexample to the naive generalization of \Cref{thm:intro-local-global} to discriminants (\Cref{prop:counterexample-discr}).

\subsection{Terminology and notation}
\label{subsn:notation}

In this article, we let $\N = \{1,2,\dots\}$.
All finite groups are equipped with the discrete topology, and infinite Galois groups are equipped with the Krull topology.
If~$\Gamma, G$ are topological groups, $\Hom(\Gamma, G)$ is the set of \emph{continuous} group homomorphisms~$\Gamma \to G$ (which is an abelian group if~$G$ is abelian).
Similarly, the sets~$H^i(\Gamma, G)$ are defined using continuous group cohomology.
They are abelian groups if~$G$ is abelian.
If~$G$ is non-abelian, they are defined only if $i \in \{0,1\}$, and are pointed sets with no natural group structure, cf.~\cite[Chap.~VII, Annexe]{serrecl}.
By default, if~$G$ is a finite group, we assume that~$\Gamma_K$ acts trivially on~$G$, so~$H^1(\Gamma_K, G)$ is the set of $G$-conjugacy classes of continuous group homomorphisms~$\Gamma_K \to G$, corresponding to $G$-extensions of~$K$ in the sense of \cite[Subsection~2.1]{wildcount}.

If~$K$ is a global or a local field, we let~$K^\ur$ be the maximal unramified extension of~$K$ inside~$K^\sep$.
If~$G$ is a topological group, we let $\Hom_\ur(\Gamma_K, G) := \Hom(\Gal(K^\ur|K), G)$, which we identify with the subset of~$\Hom(\Gamma_K, G)$ consisting of those homomorphisms that are trivial on~$\Gamma_{K^\ur} \subseteq \Gamma_K$.
(When~$K$ is a local field, $\Gamma_{K^\ur}$ is the inertia subgroup of~$\Gamma_K$.)

For any local field $K$ and any real number $v \geq -1$, we let~$\Gamma_K^v$ be the $v$-th higher inertia subgroup of~$\Gamma_K$ in the upper numbering (see \cite[Chap.~IV, \S3]{serrecl}).
For instance, $\Gamma_K^0$ is the inertia subgroup of~$\Gamma_K$ and~$\Gamma_K^1$ is the wild inertia subgroup.

For each prime~$\fp$ of a global field~$K$, we fix a prime~$\mathfrak P$ of~$K^\sep$ above~$\fp$ and use~$(K^\sep)_{\mathfrak P}$ as our go-to separable closure of~$K_\fp$.
We obtain an embedding $\Gamma_{K_\fp} \hookrightarrow \Gamma_K$ via $\Gamma_{K_\fp} = \Gal((K^\sep)_{\mathfrak P}|K_\fp) \simeq D_{\mathfrak P|\fp} \subseteq \Gamma_K$.
Any other choice of~$\mathfrak P$ gives rise to a conjugate embedding.
Composing a group homomorphism~$\rho \colon \Gamma_K \to G$ with any two of these conjugate embeddings yields conjugate homomorphisms $\Gamma_{K_\fp} \to G$, so any conjugation-invariant property of the restricted homomorphism~$\rho\restfp$ (e.g., its last jump) is in fact canonical.
If~$G$ is abelian, the restriction $\rho\restfp \colon \Gamma_{K_\fp} \to G$ is itself independent of the choice of~$\mathfrak P$.

\subsection{Acknowledgments}

This work was supported by the Deutsche Forschungsgemeinschaft (DFG, German Research Foundation)---Project-ID 491392403---TRR 358 (Project A4).
The authors are grateful to Jürgen Klüners, Raphael Müller and Carlo Pagano for helpful discussions.

\section{Preliminaries}
\label{sn:preliminaries}

\subsection{Embedding problems in characteristic~$p$}

Let~$K$ be a field of characteristic~$p$.
For $p$-extensions of~$K$, there are no obstructions to solving embedding problems with central kernel:

\begin{lemma}
  \label{lem:embpb-solvable}
  Consider a short exact sequence $1 \to N \to G \overset\pi\to Q \to 1$, and assume that~$N \subseteq Z(G)$ is a finite abelian $p$-group.
  Then, the natural map $\Hom(\Gamma_K, G) \to \Hom(\Gamma_K, Q)$, $\rho \mapsto \pi\circ\rho$ is surjective.
  Moreover, twisting (pointwise multiplication) defines a simply transitive action of~$\Hom(\Gamma_K, N)$ on each fiber of that surjection.
\end{lemma}

\begin{proof}
  By \cite[(6.1.2)]{nsw}, we have~$H^2(\Gamma_K,\F_p) = 0$.
  By induction on the order of~$N$ (which contains an element of order~$p$ unless $N=0$), using the long exact sequence in group cohomology, it follows that~$H^2(\Gamma_K, N) = 0$.
  The surjectivity claim then follows from the exact sequence $H^1(\Gamma_K, G) \to H^1(\Gamma_K, Q) \to H^2(\Gamma_K, N)$ of pointed sets, see \cite[Chap.~VII, Annexe, Prop.~2]{serrecl}.
  The properties of twisting are straightforward.
\end{proof}

(See \cite[Theorem~1.1]{embpb} for generalizations of \Cref{lem:embpb-solvable} to non-abelian kernels.)

\subsection{Unramified extensions}
\label{subsn:unram-ext}

The maximal unramified extension of a local function field $K = \F_q\ppar{T}$ is $K^\ur = \bar\F_q\ppar{T}$, and that of a rational global function field $K = \F_q(T)$ is $K^\ur = \bar\F_q(T)$ by the Riemann--Hurwitz formula.
In both cases, $\Gal(K^\ur|K) \simeq \Gal(\bar\F_q|\F_q) \simeq \hat\Z$ is topologically generated by the lift of the Frobenius automorphism $x \mapsto x^q$ of $\bar\F_q|\F_q$, so any continuous unramified homomorphism $\Gamma_K\to G$ is determined by the image of this generator.
Therefore,
\begin{equation}
  \label{eq:unram-ext}
  \card{\Hom_\ur(\Gamma_K, G)} = \card G,
\end{equation}
and for any normal subgroup $N$ of $G$, the natural map
\begin{equation}
  \label{eq:unram-lift}
  \Hom_\ur(\Gamma_K,G)
  \twoheadrightarrow
  \Hom_\ur(\Gamma_K,G/N)
\end{equation}
is surjective.

We also make explicit the way in which the local and global identifications of~$\hat\Z$ with the Galois group of the maximal unramified extension are related:

\begin{lemma}
  \label{inclusion-frobenius}
  Let $K = \F_q(T)$, and let $\fp \in \cP_K$.
  Under the identification of both sides with~$\hat\Z$, the restriction map $\Gal(K_\fp^\ur|K_\fp) \hookrightarrow \Gal(K^\ur|K)$ coincides with multiplication by~$\deg\fp$.
\end{lemma}

\begin{proof}
  Both sides are identified with~$\hat\Z$ by looking at how they act on~$\bar\F_q$: the (topological) generator of~$\Gal(K_\fp^\ur|K_\fp)$ is the unique lift $\tilde\sigma_{K_\fp}$ of $x \mapsto x^{N(\fp)}$, and the generator of~$\Gal(K^\ur|K)$ is the unique lift of $\tilde\sigma_K \colon x \mapsto x^q$.
  The result follows as $N(\fp) = q^{\deg\fp}$.
\end{proof}

\subsection{Local--global principles for abelian $p$-extensions}

Let~$K$ be either a local or a global field, and let~$G$ be an abelian group.

\begin{definition}
  \label{def:unramified-twist-inertial-type}
  Two homomorphisms~$\rho_1, \rho_2 \in \Hom(\Gamma_K,G)$ are \emph{unramified twists} of each other if they coincide when restricted to~$\Gamma_{K^\ur}$, or equivalently when~$\rho_1 - \rho_2 \in \Hom_\ur(\Gamma_K,G)$.
  We call \emph{inertial types} the equivalence classes for this relation, i.e., the elements of the quotient group
  \[
    \cI(K,G) :=
    \bigslant
      {\Hom(\Gamma_K,G)}
      {\Hom_\ur(\Gamma_K,G)}.
  \]
\end{definition}

If $K$ is a local function field or a rational global function field, \Cref{eq:unram-ext} shows that each inertial type contains exactly $\card G$ homomorphisms.

If~$K$ is a global function field, then restriction to $\Gamma_{K_\fp}$ for all~$\fp \in \cP_K$ induces a homomorphism
\begin{equation}
  \label{def:alpha}
  \cI(K, G) \to
  \bigoplus_{\fp \in \cP_K}
    \cI(K_\fp, G)
\end{equation}
which is injective, as by definition of~$K^\ur$ any $\rho \in \Hom(\Gamma_K, G)$ which is unramified at all primes~$\fp \in \cP_K$ factors through~$\Gal(K^\ur|K)$.
In fact, class field theory implies the following:

\begin{lemma}
  \label{lem:cft}
  Assume that~$K = \F_q(T)$ is a rational global function field of characteristic~$p$, and that~$G$ is a finite abelian $p$-group.%
  \footnote{
    It is enough to assume that~$G$ is a finite abelian group of order coprime to~$|\F_q^\times| = q-1$.
    The condition that~$K$ is a rational function field can also be relaxed into the condition $\gcd(\card{\mathrm{Pic}^0(X)}, \card G)=1$, where~$X$ is the nonsingular projective curve over~$\F_q$ associated to~$K$ (the normalization of $\mathbb P^1_{\F_q}$ in $\mathrm{Spec}\,K$).
  }
  Then, the map from \Cref{def:alpha} is an isomorphism.
\end{lemma}

\begin{proof}
  Class field theory (see e.g.~\cite[(7.2.11) and (8.1.26)]{nsw}) describes~$\Gamma_K^\ab$ as the profinite completion of the quotient~$\mathbb A_K^\times/K^\times$, where~$\mathbb A_K^\times$ is the group of idèles, and~$\Gamma_{K_\fp}^\ab$ as the profinite completion of $K_\fp^\times \simeq \Z \times \cO_\fp^\times$ for any $\fp \in \cP_K$.
  Thus, for any finite abelian group~$G$, we have bijections $\Hom(\Gamma_K, G) \simeq \Hom(\mathbb A_K^\times/K^\times, G)$ and, for any $\fp \in \cP_K$:
  \[
    \Hom(\Gamma_{K_\fp}, G)
    \simeq
    \Hom(K_\fp^\times, G)
    \simeq
    \underbrace{\Hom(\Z, G)}_{
      \simeq \Hom_\ur(\Gamma_{K_\fp}, G)
    }
    \oplus \Hom(\cO_\fp^\times, G),
  \]
  whence $\cI(K_\fp, G) \simeq \Hom(\cO_\fp^\times, G)$.
  The idèle class group fits into the exact sequence
  \[
    0
    \to
    \F_q^\times
    \to
    \Z \times \prod_{\fp\in\cP_K} \cO_\fp^\times
    \to
    \mathbb A_K^\times/K^\times
    \to
    0,
  \]
  where the rightmost term vanishes since~$K$ is a rational function field (the associated curve is~$\mathbb P^1_{\F_q}$, with trivial~$\mathrm{Pic}^0$).
  Since~$G$ is a finite abelian $p$-group and $p \nmid {|\F_q^\times|}$, we have isomorphisms
  \begin{align*}
    \Hom(\Gamma_K, G)
    & \simeq
    \Hom(\mathbb A_K^\times/K^\times, \, G)
    \\
    & \simeq
    \Hom \leftl(
      \Z \times \prod_{\fp\in\cP_K} \cO_\fp^\times, \;
      G
    \rightr)
    \simeq
    \underbrace{\Hom(\Z, G)}_{\simeq \; \Hom_\ur(\Gamma_K, G)}
    \oplus
    \bigoplus_{\fp \in \cP_K}
      \underbrace{
        \Hom(\cO_\fp^\times, G)
      }_{\simeq \; \cI(K_\fp, G)}
  \end{align*}
  whence
  $
    \cI(K, G)
    \simeq
    \bigoplus_\fp
      \cI(K_\fp, G).
  $
  We omit the verification that this isomorphism coincides with the map from \Cref{def:alpha}.
\end{proof}

(When~$K=\F_q(T)$ and~$G=\Z/p\Z$, \Cref{lem:cft} can be proved without any class field theory by describing and comparing normal forms for elements of~$K/\wp(K)$ and of~$K_\fp/\wp(K_\fp)$, cf.~\cite[Proposition~5.2, Corollary~5.3, Remark~5.4]{potthast}. A similar proof can be made for all finite abelian $p$-groups~$G$ via Artin--Schreier--Witt theory.)

This allows for a simple proof of \Cref{thm:intro-local-global} when~$G$ is abelian: dividing the counts of homomorphisms by~$\card G$ (the number of unramified homomorphisms, cf.~\Cref{eq:unram-ext}) means that we are counting inertial types, for which we can apply the local--global bijection of \Cref{lem:cft}.%
\footnote{
  This explanation might suggest that the factor $\frac1{\card G}$ in \Cref{thm:intro-local-global} is simply identifying homomorphisms coinciding on inertia.
  However, this interpretation is incorrect for non-abelian~$G$: for a given $\rho \in \Hom(\Gamma_{K_\fp}, G)$, there are not~$\card G$ homomorphisms sharing the same restriction to inertia, but only $\card{\mathrm{Cent}_G(\rho(\Gamma_{K_\fp}^0))}$.
}

\subsection{Inertial heights}
Fix a finite group~$G$ and a field~$K$ which is either local or global.

\begin{definition}
  \label{def:inertial-height}
  An \emph{(additive) inertial height for $G$-extensions of~$K$} is a map $\height \colon \Hom(\Gamma_K, G) \to \R_{\geq0}$ satisfying the following properties:
	\begin{enumalpha}
    \item
      \label{fact:conjugation}
      For any~$\rho \in \Hom(\Gamma_K, G)$ and any $g \in G$, we have $\height \rho = \height(g \rho g^{-1})$.
    \item
      \label{fact:height-zero}
      We have $\height \rho = 0$ if and only if $\rho \in \Hom_\ur(\Gamma_K, G)$.
    \item
      \label{fact:inertial}
      The value $\height \rho$ only depends on the restriction $\rho|_{\Gamma_{K^\ur}}$ of~$\rho$ to~$\Gamma_{K^\ur}$.

      \noindent
      (When~$K$ is a local field, $\Gamma_{K^\ur}$ is the inertia subgroup~$\Gamma^0_K$.)
    \item
      \label{fact:finiteness}
      For any~$v \geq 0$, there are only finitely many $\rho \in \Hom(\Gamma_K, G)$ with $\height \rho \leq v$.

      \noindent
      (Property~\ref{fact:finiteness} is sometimes called the ``Northcott property.'')
	\end{enumalpha}
\end{definition}

Assume that~$K$ is a global function field.
For each~$\fp \in \cP_K$, let~$\height_\fp$ be a local inertial height for $G$-extensions of the completion~$K_\fp$, and assume that there is a $\mu > 0$ such that all the heights~$\height_\fp$ take values in $\{0\} \cup [\mu,+\infty)$.
Then, for any $\rho \in \Hom(\Gamma_K, G)$, the number $\height_\fp(\rho\restfp)$ is well-defined by property \ref{fact:conjugation} (cf.~\Cref{subsn:notation}), and we define a \emph{global height} $\height \colon \Hom(\Gamma_K, G) \to \R_{\geq 0}$ by
\begin{equation}
  \label{def:global-height}
  \height {\rho}
  :=
  \sum_{\fp \in \cP_K}
    \deg \fp
    \cdot
    \height_\fp \bigl( \rho\restfp \bigr).
\end{equation}
(Only finitely many primes~$\fp$ are ramified, so almost all the terms of this sum are zero by property~\ref{fact:height-zero}.
The Northcott property for the global height follows from the fact that only the primes of degree $\leq \frac v\mu$ can ramify in an extension of height $\leq v$, from the Northcott property for the local heights~$\height_\fp$, and from the fact that there are only finitely many separable extensions of~$K$ of degree~$\leq \card G$ with bounded discriminant, cf.~\cite[Theorem~8.23.5, 3.]{hermite-goss} or \cite{hermite-wong,hermite-rosen}. Therefore, $\height$ is an inertial height for $G$-extensions of~$K$.)

\medskip

Assume that~$K$ is a local field.
A classical example of an (additive) inertial height for $G$-extensions of~$K$ is the \emph{discriminant exponent}
\begin{equation}
  \label{eq:def-discr-exp}
  \delta(\rho)
  :=
  \card{G}
  \int_{-1}^{+\infty}
    \leftl(
      1
      -
      \frac1{\card {\rho(\Gamma_K^v)}}
    \rightr)
    \d v
    \quad
    \in \Z_{\geq 0},
\end{equation}
which is such that $q^{\delta(\rho)}$ is the norm of the relative discriminant of the $G$-extension~$L|K$ associated to~$\rho$ (an étale $K$-algebra of dimension~$\card G$).%
\footnote{
  This follows from \cite[Chap.~IV, \S1, Prop.~4]{serrecl} after converting to the upper numbering and taking the norm.
  (We have $f(L|K) \cdot \card{\rho(\Gamma_K^0)} = \card{\rho(\Gamma_K)}$, and multiplying the discriminants of the~$\frac{\card G}{\card{\rho(\Gamma_K)}}$ isomorphic field factors of the étale $K$-algebra $L$ yields the factor~$\card G$.)
}
Over global function fields, the associated global height is then the degree of the discriminant divisor.

\subsection{The last jump}
\label{subsn:last-jump}

Let~$K$ be a local field, and let~$G$ be a finite group.

\begin{definition}
  \label{def:local-lj}
  For any continuous map~$\rho \colon \Gamma_K \to G$ satisfying $\rho(1)=1$ (not necessarily a group homomorphism!), we define the \emph{last jump} of~$\rho$ as follows:
  \[
      \lastjump \rho
      :=
      \inf\suchthat{
          v \in \Q_{\geq 0}
      }{
			\rho(\Gamma_K^v) = 1
      }
      \in \R_{\geq 0}.
  \]
\end{definition}

This number $\lastjump \rho$ is indeed finite:
If $\lastjump \rho = +\infty$, then the nested closed subsets $\suchthat{\tau \in \Gamma_K^v}{\rho(\tau)\neq1}$ of $\Gamma_K$ are non-empty for all $v \geq 0$.
Since~$\Gamma_K$ is compact Hausdorff, their intersection is also non-empty, hence contains some $\tau \in \bigcap_{v \geq 0} \Gamma_K^v = 1$, contradicting $\rho(1)=1$.

Note that~$\lastjump \rho$ is unaffected when composing~$\rho$ with an automorphism of~$G$, e.g., under conjugation by an element of~$G$.
Moreover, $\lastjump \rho$ only depends on the restriction of~$\rho$ to the inertia subgroup~$\Gamma_K^0$.
For any two continuous maps~$\rho, \rho' \colon \Gamma_K \to G$ with $\rho(1)=\rho'(1)=1$, denoting their pointwise product by~$\rho \cdot \rho'$, we have the ``ultrametric inequality''
\begin{equation}
    \label{eqn:ultineq}
    \lastjump {(\rho \cdot \rho')}
    \leq
    \max(
        \lastjump \rho, \,
        \lastjump \rho'
    )
\end{equation}
with equality whenever~$\lastjump \rho \neq \lastjump \rho'$.

If~$\rho$ is a group homomorphism, then $\lastjump \rho = 0$ if and only if~$\rho$ is at most tamely ramified, i.e., if $\rho(\Gamma_K^1)=1$.
When~$G$ is a $p$-group and~$K$ has residue characteristic~$p$, this amounts to~$\rho$ being unramified, and in that case $\lastjump \colon \Hom(\Gamma_K, G) \to \Q_{\geq0}$ is an (additive, local) inertial height.
If moreover~$G$ is abelian, then the last jump is related to the conductor exponent by the formula $\lastjump\rho = \max(0, \, \cond\rho - 1)$, and in particular it is an integer.
For abelian $G$, this height is also sometimes called the \emph{Artin--Schreier conductor}.

\subsection{Reminders on Artin--Schreier--Witt theory}
\label{subsn:asw}

In this section, we recall classical facts about the explicit parametrization of finite abelian $p$-extensions in characteristic~$p$.
Let $K$ be a field of characteristic~$p$, and consider a finite abelian $p$-group
\[
  G \simeq \prod_{i=1}^r \Z/p^{n_i}\Z.
\]
Although the class field theoretic description of~$\Gamma_K^\ab$ leads to a description of~$\Hom(\Gamma_K, G)$ for local and global fields, we will need the explicit parametrization provided by Artin--Schreier--Witt theory.
We denote the ring of Witt vectors over $K$ by $W(K)$ and the ring of Witt vectors of length $n$ by $W_n(K)$.
Let~$K^\perf$ (resp.~$K^\alg$) be the perfect closure of~$K$ (resp.~of~$K^\sep$), cf.~\cite[Subsection~2.3]{wildcount}.
See~$G$ as a finite $\Z_p$-module, and define the $W(K^\perf)$-module
\[
  G_K := G \otimes_{\Z_p} W(K^\perf) \simeq \prod_{i=1}^r W_{n_i}(K^\perf).
\]
(The isomorphism $\Z/p^{n_i}\Z\otimes W(K^\perf) \simeq W_{n_i}(K^\perf)$ relies on the fact that~$K^\perf$ is a perfect field. This is our reason not to instead use~$G\otimes_{\Z_p} W(K)$.)
Any endomorphism of~$K$ induces an endomorphism of~$G_K$.
For instance, the absolute Frobenius~$\sigma \colon x \mapsto x^p$ acts on~$G_K$, and its fixed points are exactly the elements of~$G\otimes_{\Z_p} W(\F_p) = G\otimes_{\Z_p}\Z_p = G$.
Similarly, for any Galois extension~$L|K$, the group $\Gal(L|K)$ acts on~$G_L$ and the fixed submodule is exactly~$G_K$.
Moreover, $H^1(\Gamma_K, G_{K^\sep})$ is the trivial group (this follows from the additive version of Hilbert's Theorem 90).
For reminders about Witt vectors and for the missing details, see~\cite[Section~4.10]{bosch-algebra} and/or \cite[Subsection~2.4]{wildcount}.

We define the $\Z_p$-linear map $\wp \colon G_K \to G_K$, $x \mapsto \sigma(x) - x$, with $\ker\wp = G$.
The map~$\wp$ is surjective when~$K$ is separably closed.
As in Artin--Schreier theory, the short exact sequence
\[
  0
  \longrightarrow
  G
  \longrightarrow
  G_{K^\sep}
  \overset\wp\longrightarrow
  G_{K^\sep}
  \to
  0
\]
induces, in Galois cohomology, an exact sequence
\[\begin{tikzcd}[row sep=0.7em]
	{H^0(\Gamma_K, G_{K^\sep})} & {H^0(\Gamma_K, G_{K^\sep})} & {H^1(\Gamma_K, G)} & {H^1(\Gamma_K,G_{K^\sep})} \\
	{G_K} & {G_K} & {\Hom(\Gamma_K,G)} & 0
	\arrow[from=1-1, to=1-2]
	\arrow[from=1-2, to=1-3]
	\arrow[from=1-3, to=1-4]
	\arrow["{=}"{marking, allow upside down}, draw=none, from=2-1, to=1-1]
	\arrow["\wp", from=2-1, to=2-2]
	\arrow["{=}"{marking, allow upside down}, draw=none, from=2-2, to=1-2]
	\arrow[from=2-2, to=2-3]
	\arrow["{=}"{marking, allow upside down}, draw=none, from=2-3, to=1-3]
	\arrow[from=2-3, to=2-4]
	\arrow["{=}"{marking, allow upside down}, draw=none, from=2-4, to=1-4]
\end{tikzcd}\]
whence an isomorphism $\Hom(\Gamma_K, G) \simeq G_K/\wp(G_K)$.
Explicitly, the $\rho \in \Hom(\Gamma_K,G)$ corresponding to a coset $[m] \in G_K/\wp(G_K)$ is obtained by picking any~$g \in \wp^{-1}(m)\subseteq G_{K^\sep}$ and by letting~$\rho(\tau) := \tau(g) - g$.
For details, one can check \cite[Corollary~2.13]{wildcount}.

\subsubsection{Local Artin--Schreier extensions}
Assume now that~$K = \F_q\ppar{T}$ is a local function field of characteristic~$p$.
Recall that~$G$ is a finite abelian $p$-group.
Let~$[T] \in W(K^\perf)$ be the Teichmüller lift of the uniformizer~$T$, and let~$\cD^0$ be the sub-$W(\F_q)$-module of~$W(K^\perf)$ spanned by the elements~$[T]^{-n}$ for~$n \in \Onotp$.
By \cite[Lemma~3.4~(ii)]{wildcount} (see also \cite[Theorem~3~(i)]{kanesaka-sekiguchi}), every coset in~$G_K/\wp(G_K)$ intersects the following $W(\F_q)$-submodule of~$G_K$:
\[
  G \otimes_{\Z_p} \cD^0
  =
  \suchthat{
    m_0
    +
    \sum_{n \in \notp}
      m_n [T]^{-n}
  }{
    m_n \in G_{\F_q} = G \otimes_{\Z_p} W(\F_q) \textnormal{ for all } n \in \Onotp \\
    m_n = 0 \textnormal{ for almost all } n
  }.
\]
This leads to a refinement of the parametrization of elements of $\Hom(\Gamma_K, G)$:

\begin{proposition}
  \label{prop:local-fundom}
  The isomorphism $\Hom(\Gamma_K, G) \simeq G_K/\wp(G_K)$ described above induces group isomorphisms
  \[
    \Hom(\Gamma_K, G)
    \simeq
    \bigslant{G\otimes\cD^0}{\wp(G_{\F_q})}
    \qquad
    \textnormal{ and }
    \qquad
    \cI(K, G)
    \simeq
    \bigslant{G\otimes\cD^0}{G_{\F_q}}.
  \]
  Moreover, the subgroup $\Hom_\ur(\Gamma_K, G)$ of $\Hom(\Gamma_K, G)$ corresponds to $G_{\F_q}/\wp(G_{\F_q})$.
\end{proposition}

In other words, any given homomorphism~$\rho \in \Hom(\Gamma_K, G)$ corresponds to an element~$m$ that we can pick in~$G \otimes \cD^0$, and the choice of~$m$ is almost unique: the elements~$m_n$ for~$n \in \notp$ are uniquely determined, and together they characterize the inertial type of~$\rho$, whereas~$m_0$ is only unique modulo~$\wp(G_{\F_q})$.
Note that $G_{\F_q} = G\otimes W(\F_q)$ is a finite group.
For more details, see e.g.~\cite[Subsection~3.1]{wildcount}.

Finally, we recall the description of the last jump for abelian $p$-extensions of local function fields of characteristic~$p$ in terms of this refined parametrization:

\begin{proposition}[{{\cite[Theorem~5~(ii)]{kanesaka-sekiguchi}}}]
  \label{prop:lj-asw}
  Let~$\rho \in \Hom(\Gamma_K, G)$, and let~$m \in G \otimes \cD^0$ be such that~$\rho(\tau) = \tau(g) - g$ for any~$g \in \wp^{-1}(m)$.
  Then, writing~$m$ as
  \[
    m
    =
    m_0 +
    \sum_{n \in \notp}
      m_n
      [T]^{-n}
    \qquad
    \textnormal{with }
    m_n \in G \otimes W(\F_q)
    \textnormal{ almost all zero}
  \]
  and letting~$\mu_v(n) := \cardsuchthat{k \geq 0}{np^k < v}$, we have
  \begin{equation}
    \label{eqn:formula-lj-asw}
    \lastjump \rho
    =
    \min\suchthat{
      v \in \Z_{\geq 0}
    }{
      \forall n \in \notp,\,
      p^{\mu_{v+1}(n)} m_n = 0
    }
    \quad
    \in \Z_{\geq 0}.
  \end{equation}
\end{proposition}

\begin{proof}
  We can restrict our attention to a single cyclic factor of~$G$, i.e., we can assume that~$G = \Z/p^N\Z$, and then~%
  $
    G \otimes W(\F_q)
    = W_N(\F_q)
  $.
  In the notation of~\cite{kanesaka-sekiguchi}, for any~$n \in \notp$ such that~$m_n \neq 0$, the integer ``$l_n$'' is~$N - v_p(m_n)$ (where~$v_p$ is the $p$-adic valuation).
  \cite[Theorem~5~(ii)]{kanesaka-sekiguchi} implies the following formula for the last jump of~$\rho$ (which is zero if~$\rho$ is unramified, and one less than the conductor otherwise):
  \[
    \lastjump \rho =
    \max\leftl(
      \{0\} \cup
      \suchthat{
        n p^{l_n - 1}
      }{
        n \in \notp \\
        m_n \neq 0
      }
    \rightr).
  \]
  In other words, for any $v \geq 0$, we have $\lastjump \rho \leq v$ if and only if $n p^{l_n - 1} \leq v$ for all~$n \in \notp$ with~$m_n \neq 0$.
  Fix~$n \in \notp$.
  Note that, by definition, for any~$v, k \in \N$, we have the equivalences
  \[
    k \leq \mu_{v+1}(n) \Leftrightarrow k-1 < \mu_{v+1}(n) \Leftrightarrow n p^{k-1} < v+1 \Leftrightarrow n p^{k-1} \leq v,
  \]
  so, for any~$k \in \N$ and~$n \in \notp$:
  \begin{align*}
    n p^{l_n - 1} \leq v \iff l_n \leq \mu_{v+1}(n) & \iff v_p(m_n) \geq N - \mu_{v+1}(n)
    \\
    & \iff v_p(p^{\mu_{v+1}(n)} m_n) \geq N \iff p^{\mu_{v+1}(n)} m_n = 0.
    & \qedhere
  \end{align*}
\end{proof}

\begin{remark}
  \label{rk:elemab}
  When~$G = \F_p$, we recover a form of Artin--Schreier theory.
  We have $G_K = \F_p\otimes W(K^\perf) = K^\perf$ and $G_{\F_q} = \F_p\otimes W(\F_q) = \F_q$, so
  \[
    G \otimes \cD^0
    = \F_p \otimes \cD^0
    =
    \suchthat{
      m_0
      +
      \sum_{n \in \notp}
        m_n T^{-n}
    }{
      m_n \in \F_q \textnormal{ for all } n \in \Onotp \\
      m_n = 0 \textnormal{ for almost all } n
    }
    \subseteq \F_q\ppar{T}.
  \]
  \Cref{prop:local-fundom} then says that we have isomorphisms
  \[
    \Hom(\Gamma_K, \F_p)
    \simeq
    \bigslant{\F_p\otimes\cD^0}{\wp(\F_q)}
    \qquad
    \textnormal{ and }
    \qquad
    \cI(K, \F_p)
    \simeq
    \bigslant{\F_p\otimes\cD^0}{\F_q},
  \]
  the subgroup $\Hom_\ur(\Gamma_K, \F_p)$ of $\Hom(\Gamma_K, \F_p)$ corresponds to $\F_q/\wp(\F_q)$,
  and \Cref{eqn:formula-lj-asw} becomes (letting $v_T$ denotes the $T$-adic valuation)
  \[
    \lastjump \rho = \min\suchthat{v \in \Z_{\geq0}}{\forall n > v, m_n=0} = \max(0, -v_T(m))
  \]
  which in particular always belongs to~$\Onotp$.
\end{remark}

\section{An abstract local--global principle}
\label{sn:locglob}

In this section, let~$K := \F_q(T)$ be a rational global function field of characteristic~$p$, consider a short exact sequence
\[
  1 \longrightarrow N \longrightarrow G \overset\pi\longrightarrow Q \longrightarrow 1
\]
of finite $p$-groups with~$N \subseteq Z(G)$ and with~$Q$ abelian, and fix for each $\fp \in \cP_K$ an (additive, local) inertial height $\height_\fp \colon \Hom(\Gamma_{K_\fp}, G) \to \R_{\geq 0}$ for $G$-extensions of~$K_\fp$.
The main result of this section is the following abstract local-global principle, which will allow us to deduce \Cref{thm:intro-local-global} from \Cref{thm:minembed-inertial}:

\begin{theorem}
  \label{thm:abstract-locglob}
  Assume that for each $\fp \in \cP_K$ the inertial height~$\height_\fp$ satisfies the following property: for any $\bar\rho_1, \bar\rho_2 \in \Hom(\Gamma_{K_\fp},Q)$ with $\bar\rho_1 - \bar\rho_2 \in \Hom_\ur(\Gamma_{K_\fp}, Q)$ and for any $v \geq 0$, we have
  \begin{equation}
    \label{fact:embed-inertial}
    \cardsuchthat{
      \rho \in \Hom(\Gamma_{K_\fp}, G)
    }{
      \pi \circ \rho = \bar\rho_1 \\
      \height_\fp {(\rho)} = v
    }
    =
    \cardsuchthat{
      \rho \in \Hom(\Gamma_{K_\fp}, G)
    }{
      \pi \circ \rho = \bar\rho_2 \\
      \height_\fp {(\rho)} = v
    }
    \tag{$\star$}
  \end{equation}
  Then, for any $(v_\fp) \in \prod_{\fp \in \cP_K} \R_{\geq0}$ such that~$v_\fp = 0$ for almost all~$\fp$, we have:
  \begin{align*}
    &\frac{1}{\card G}
    \cardsuchthat{
      \rho \in \Hom(\Gamma_K, G)
    }{
      \forall \fp,\ \height_\fp \bigl(\rho\restfp\bigr) = v_\fp
    }
    \\={}&
    \prod_{\fp \in \cP_K}
      \frac{1}{\card G}
      \cardsuchthat{
        \rho_\fp \in \Hom(\Gamma_{K_\fp}, G)
      }{
        \height_\fp (\rho_\fp) = v_\fp
      },
  \end{align*}
  where the left-hand side and all the factors of the right-hand side are nonnegative integers.
\end{theorem}

\begin{remark}
  By property \ref{fact:conjugation} of inertial heights (see \Cref{def:inertial-height}), the left-hand side is independent of the embedding $\Gamma_{K_\fp}\hookrightarrow\Gamma_K$ corresponding to the choice of a prime~$\mathfrak P$ of~$K^\sep$ above~$\fp$.
  By property \ref{fact:height-zero} and \Cref{eq:unram-ext}, for primes~$\fp$ with~$v_\fp=0$, the factor is~$1$.
  By property \ref{fact:finiteness}, each factor is finite.
  Hence, the infinite product is well-defined.
\end{remark}

We first deal with individual fibers of the surjection $\Hom(\Gamma_K, G) \twoheadrightarrow \Hom(\Gamma_K, Q)$:

\begin{lemma}
  \label{lem:fiberwise-local-global}
  Fix a $\bar\rho \in \Hom(\Gamma_K, Q)$.
  For any~$(v_\fp) \in \prod_{\fp\in\cP_K} \R_{\geq0}$ such that $v_\fp=0$ for almost all~$\fp$, we have
  \begin{align*}
    &\frac{1}{\card N}
    \cardsuchthat{
      \rho \in \Hom(\Gamma_K, G)
    }{
      \pi\circ\rho = \bar\rho \\
      \forall \fp,\ \height_\fp \bigl(\rho\restfp\bigr) = v_\fp
    }
    \\={}&
    \prod_{\fp \in \cP_K}
      \frac{1}{\card N}
      \cardsuchthat{
        \rho_\fp \in \Hom(\Gamma_{K_\fp}, G)
      }{
        \pi\circ\rho_\fp = \bar\rho\restfp \\
        \height_\fp {(\rho_\fp)} = v_\fp
      },
  \end{align*}
  where the left-hand side and all the factors of the right-hand side are nonnegative integers.
\end{lemma}

\begin{proof}
  We first show that the restriction map
  \begin{align*}
    &\bigslant
      {\suchthat{
        \rho \in \Hom(\Gamma_K, G)
      }{
        \pi\circ\rho=\bar\rho
      }}
      {
        \Hom_\ur(\Gamma_K, N)
      }
    \\
    \longrightarrow&
    \prod_{\fp \in \cP_K}\prim
    \Bigl(
      \bigslant
        {\suchthat{
          \rho_\fp \in \Hom(\Gamma_{K_\fp}, G)
        }{
          \pi\circ\rho_\fp=\bar\rho\restfp
        }}
        {
          \Hom_\ur(\Gamma_{K_\fp}, N)
        }
    \Bigr)
  \end{align*}
  is a bijection (the restricted product means that~$\rho_\fp$ must be unramified for almost all~$\fp$).
  Indeed, use \Cref{lem:embpb-solvable} to pick any~$\rho_0 \in \Hom(\Gamma_K, G)$ such that~$\bar\rho = \pi\circ\rho_0$.
  The other lifts of~$\bar\rho$ (resp.~of $\bar\rho\restfp$) are the twists of~$\rho_0$ (resp.~of $\rho_0\restfp$) by elements $\delta \in \Hom(\Gamma_K, N)$ (resp.~$\delta_\fp \in \Hom(\Gamma_{K_\fp}, N)$).
  Via this description, the claim boils down to checking that the restriction map $\cI(K, N) \to \bigoplus_\fp \cI(K_\fp, N)$ is a bijection, which is exactly \Cref{lem:cft}.

  By property~\ref{fact:inertial} of inertial heights, $\height_\fp$ is constant on each $\Hom_\ur(\Gamma_K, N)$-orbit (resp.~on each $\Hom_\ur(\Gamma_{K_\fp}, N)$-orbit), so the bijection above (the restriction map) induces a bijection
  \begin{align*}
    &\bigslant
      {\suchthat{
        \rho \in \Hom(\Gamma_K, G)
      }{
        \pi\circ\rho=\bar\rho \\
        \forall \fp,\ \height_\fp \bigl(\rho\restfp\bigr) = v_\fp
      }}
      {
        \Hom_\ur(\Gamma_K, N)
      }
    \\
    \simeq&
    \prod_{\fp \in \cP_K}\prim
    \leftl(
      \bigslant
        {\suchthat{
          \rho_\fp \in \Hom(\Gamma_{K_\fp}, G)
        }{
          \pi\circ\rho_\fp=\bar\rho\restfp \\
          \height_\fp {(\rho_\fp)} = v_\fp
        }}
        {
          \Hom_\ur(\Gamma_{K_\fp}, N)
        }
    \rightr).
  \end{align*}

  If there is a~$\fp \in \cP_K$ with $v_\fp = 0$ such that~$\bar\rho$ is ramified at~$\fp$, then the condition~$\height_\fp \bigl(\rho\restfp\bigr) = v_\fp$ (resp.~$\height_\fp {(\rho_\fp)} = v_\fp$) is impossible for lifts of~$\bar\rho$ (resp.~of $\bar\rho\restfp$), so both sides are empty and the result is obvious.
  Further, we assume that~$\bar\rho$ is unramified at all primes $\fp \in \cP_K$ with $v_\fp = 0$.

  In the restricted product, the primes~$\fp$ with~$v_\fp=0$ can then be ignored: the corresponding~$\rho_\fp$ are the unramified lifts of~$\bar\rho\restfp$ (which is unramified) by property~\ref{fact:height-zero} of inertial heights, hence a single $\Hom_\ur(\Gamma_{K_\fp},N)$-orbit.
  Hence, ignoring the factors which are singletons, the restricted product is simply a product over the finitely many primes~$\fp \in \cP_K$ with $v_\fp \neq 0$, for which the corresponding factor is a finite set by the Northcott property.
  Hence, we have a bijection between two finite sets.
  The equality between their sizes is precisely the desired equality (the groups $\Hom_\ur(\Gamma_K, N)$ and $\Hom_\ur(\Gamma_{K_\fp}, N)$ have size~$\card N$ by \Cref{eq:unram-ext}, and the twisting action is free).
  The left-hand side and the factors of the right-hand side are nonnegative integers because they are counts of $\Hom_\ur(\Gamma_K, N)$-orbits and $\Hom_\ur(\Gamma_{K_\fp}, N)$-orbits, respectively.
\end{proof}

\begin{proof}[Proof of \Cref{thm:abstract-locglob}]
  Splitting up the count into fibers, the left-hand side becomes
  \begin{align*}
    &
    \frac{1}{\card G}
    \sum_{
      \bar\rho \in \Hom(\Gamma_K, Q)
    }
      \;
      \cardsuchthat{
        \rho \in \Hom(\Gamma_K, G)
      }{
        \pi\circ\rho = \bar\rho\\
        \forall \fp,\ \height_\fp \bigl(\rho\restfp\bigr) = v_\fp
      }
    \\ ={} &
    \frac{1}{\card Q}
    \sum_{
      \bar\rho \in \Hom(\Gamma_K, Q)
    }
      \,\,
      \prod_{\fp \in \cP_K}
        \frac{1}{\card N}
        \cardsuchthat{
          \rho_\fp \in \Hom(\Gamma_{K_\fp}, G)
        }{
          \pi\circ\rho_\fp = \bar\rho\restfp \\
          \height_\fp (\rho_\fp) = v_\fp
        }
    &\textnormal{by \Cref{lem:fiberwise-local-global}}
  \end{align*}
  By the additional assumption~(\ref{fact:embed-inertial}) of \Cref{thm:abstract-locglob}, each summand depends only on the collection of inertial types $\bigl([\bar\rho\restfp]\bigr) \in \bigoplus_\fp \cI(K_\fp, Q)$, and hence only on the inertial type $[\bar\rho] \in \cI(K,Q)$ of~$\bar\rho$.
  Moreover, each inertial type contains exactly~$\card Q$ homomorphisms by \Cref{eq:unram-ext}, so the left-hand side equals the following sum over inertial types:
  \begin{align*}
    &\sum_{
      [\bar\rho] \in \cI(K, Q)
    }
      \,\,
      \prod_{\fp \in \cP_K}
        \frac{1}{\card N}
        \cardsuchthat{
          \rho_\fp \in \Hom(\Gamma_{K_\fp}, G)
        }{
          \pi\circ\rho_\fp = \bar\rho\restfp \\
          \height_\fp (\rho_\fp) = v_\fp
        }
    \\={}&
    \prod_{\fp \in \cP_K}
      \,\,
      \sum_{
        [\bar\rho_\fp] \in \cI(K_\fp, Q)
      }
          \frac{1}{\card N}
          \cardsuchthat{
            \rho_\fp \in \Hom(\Gamma_{K_\fp}, G)
          }{
            \pi\circ\rho_\fp = \bar\rho_\fp \\
            \height_\fp (\rho_\fp) = v_\fp
          }
      &\textnormal{by \Cref{lem:cft}}
  \end{align*}
  where the right-hand side again makes sense because of the assumption~(\ref{fact:embed-inertial}).
  By \Cref{lem:fiberwise-local-global}, each summand of each factor is a nonnegative integer (and hence, so are the factors).
  Finally, we obtain
  \[
    \prod_{\fp \in \cP_K}
      \,\,
      \frac1{\card Q}
      \sum_{
        \bar\rho_\fp \in \Hom(\Gamma_{K_\fp}, Q)
      }
          \frac{1}{\card N}
          \cardsuchthat{
            \rho_\fp \in \Hom(\Gamma_{K_\fp}, G)
          }{
            \pi\circ\rho_\fp = \bar\rho_\fp \\
            \height_\fp (\rho_\fp) = v_\fp
          }
  \]
  which is the right-hand side of \Cref{thm:abstract-locglob} split into fibers.
\end{proof}

\section{Lifting unramified twists}
\label{sn:minembed-urtwist}

In this section, we fix a local function field~$K = \F_q\ppar{T}$ of characteristic~$p$ and a short exact sequence
\[
  1 \longrightarrow N \longrightarrow G \stackrel\pi\longrightarrow Q \longrightarrow 1
\]
of finite $p$-groups with~$N \subseteq Z(G)$ and~$Q$ abelian.
Like in \Cref{subsn:asw}, seeing~$N$ and~$Q$ as finite $\Z_p$-modules, we define the $W(K^\perf)$-modules $N_K := N \otimes_{\Z_p} W(K^\perf)$ and $Q_K := Q \otimes_{\Z_p} W(K^\perf)$, and similarly the $W(K^\alg)$-modules~$N_{K^\sep} := N \otimes W(K^\alg)$ and~$Q_{K^\sep} := Q \otimes W(K^\alg)$.

Our goal is to prove \Cref{thm:minembed-inertial}.
In \Cref{subsn:minlift}, we define the \emph{minimal lift height}.
In \Cref{subsn:commutators}, we extend the commutator bracket into an alternating $W(K^\alg)$-bilinear map $[-,-] \colon Q_{K^\sep}^2 \to N_{K^\sep}$.
Finally, in \Cref{subsn:proof-minembed-inertial}, we carry out the main construction of this article (\Cref{lem:epsilon-minembed,lem:explicit-epsilon}), allowing us to prove \Cref{thm:minembed-inertial}.

For any $\bar\rho \in \Hom(\Gamma_K, Q)$ (resp.~$\bar\rho \in \Hom(\Gamma_{K_\fp}, Q)$), a \emph{lift of~$\bar\rho$} will always mean a homomorphism $\rho \in \Hom(\Gamma_K, G)$ (resp.~$\rho \in \Hom(\Gamma_{K_\fp}, G)$) such that $\bar\rho = \pi\circ\rho$.
By \Cref{lem:embpb-solvable}, any~$\bar\rho$ admits a lift, and all lifts are obtained by twisting a fixed lift by homomorphisms valued in~$N$.

In this section, if~$\varepsilon \colon \Gamma_K \to N$ is a continuous map, we let $\partial \varepsilon \colon \Gamma_K^2 \to N$ be the $2$-coboundary
\[
  \partial \varepsilon(\tau_1, \tau_2)
  :=
  \varepsilon(\tau_1) + \varepsilon(\tau_2) - \varepsilon(\tau_1 \tau_2)
\]
so that~$\partial \varepsilon = 0$ if and only if~$\varepsilon \in \Hom(\Gamma_K, N)$.

\subsection{Minimal lifts}
\label{subsn:minlift}

Fix a homomorphism~$\bar\rho \in \Hom(\Gamma_K, Q)$.

\begin{definition}
    The \emph{minimal lift height of~$\bar\rho$} (with respect to the surjection $\pi \colon G \twoheadrightarrow Q$) is
    \[
        \minlift {\bar\rho}
        :=
        \min\suchthat{
          \lastjump \rho
        }{
          \rho \in \Hom(\Gamma_K, G)\\
          \pi\circ\rho = \bar\rho
        }
        \in \Q_{\geq 0}.
    \]
\end{definition}

The key to proving \Cref{thm:minembed-inertial} will be to show that $\minlift \bar\rho$ depends only on the inertial type of~$\bar\rho$.

\begin{proposition}
  \label{prop:trivial-bound}
  We have $\minlift\bar\rho \geq \lastjump\bar\rho$.
  If $\lastjump\bar\rho = 0$, then $\minlift\bar\rho = 0$.
\end{proposition}

\begin{proof}
  For any lift~$\rho \in \Hom(\Gamma_K, G)$ of~$\bar\rho$ and for any $v \geq 0$, if~$\rho(\Gamma_K^v) = 1$, then $\bar\rho(\Gamma_K^v)=\pi(1)=1$.
  Therefore, $\minlift\bar\rho \geq \lastjump\bar\rho$.

  If $\lastjump\bar\rho = 0$, then~$\bar\rho$ is unramified, so it has an unramified lift $\rho$ by \Cref{eq:unram-lift}, which satisfies $\lastjump\rho=0$.
\end{proof}

We say that~$\rho$ is a \emph{minimal lift} of~$\bar\rho$ if $\lastjump \rho = \minlift {\bar\rho}$.
The following lemma states some special properties of minimal lifts:

\begin{lemma}
  \label{lem:prop-min-lift}
  Let~$\rho$ be a lift of~$\bar\rho$.
  \begin{enumroman}
    \item
      \label{lj-twists-min}
      If~$\rho$ is a minimal lift, then for any $\delta \in \Hom(\Gamma_K, N)$ we have
      \[
        \lastjump {(\delta \cdot \rho)}
        =
        \max\leftl(
          \lastjump \delta, \,
          \lastjump \rho
        \rightr).
      \]
    \item
      \label{not-integral-is-minimal}
      If $\lastjump \rho$ is not an integer, or if $N$ is elementary abelian and $\lastjump \rho$ is a multiple of~$p$, then~$\rho$ is a minimal lift.
  \end{enumroman}
\end{lemma}

\begin{proof}\
  \begin{enumroman}
  \item
  If $\lastjump \rho \neq \lastjump \delta$, then this directly follows from \Cref{eqn:ultineq}, so we assume that $\lastjump \delta = \lastjump \rho$.
  \Cref{eqn:ultineq} then implies that $\lastjump {(\delta\cdot\rho)} \leq \lastjump\rho$.
  Since $\delta\cdot\rho$ is a lift of $\bar\rho$, we have by definition of the minimal lift height that $\lastjump {(\delta\cdot\rho)} \geq \minlift{\bar\rho} = \lastjump \rho$.
  
  \item
  Write $\rho = \delta \cdot \rho_0$ for some minimal lift~$\rho_0$ and some~$\delta \in \Hom(\Gamma_K, N)$ (cf.~\Cref{lem:embpb-solvable}).
  If~$\rho$ is not minimal, then by \ref{lj-twists-min} we have $\lastjump \rho = \lastjump \delta$.
  However, since~$N$ is abelian (resp.~elementary abelian), the last jump of~$\delta$ can only be an integer (resp.~belong to $\Onotp$), cf.~\Cref{prop:lj-asw} (resp.~\Cref{rk:elemab}).
  \qedhere
  \end{enumroman}
\end{proof}

The following remark, about cohomological obstructions to embedding problems with restricted last jump, will not be used in the paper:

\begin{remark}
  We can interpret the question of determining~$\minlift \bar\rho$ as the study of embedding problems not for representations of~$\Gamma_K$ (for which there is no obstruction by \Cref{lem:embpb-solvable}), but for representations of the quotients~$\Gamma_K/\Gamma_K^v$.
  For any~$v \geq 0$, the goal is then to say something about the image of the projection map $H^1(\Gamma_K/\Gamma_K^v, G) \to H^1(\Gamma_K/\Gamma_K^v, Q)$.
  By \cite[Chap.~VII, Annexe, Prop.~2]{serrecl}, we have an exact sequence of pointed sets:
  \[
    H^1(\Gamma_K/\Gamma_K^v, G)
    \overset{\pi\circ-}\longrightarrow
    H^1(\Gamma_K/\Gamma_K^v, Q)
    \overset{\Delta}\longrightarrow
    H^2(\Gamma_K/\Gamma_K^v, N).
  \]
  Hence, a given~$\bar\rho \in \Hom(\Gamma_K, Q)$ with $\lastjump \bar\rho < v$ satisfies $\minlift \bar\rho < v$ if and only if~$\Delta$ vanishes at the corresponding element of~$H^1(\Gamma_K/\Gamma_K^v, Q)$.
  More concretely, pick a $2$-cocycle $c \colon Q^2 \to N$ representing the class $[c] \in H^2(Q,N)$ associated to the extension~$G$, and let $\Delta_c(\bar\rho)(\tau_1, \tau_2) := c\bigl(\bar\rho(\tau_1), \bar\rho(\tau_2)\bigr)$.
  We know that the $2$-cocycle~$\Delta_c(\bar\rho)$ is cohomologically trivial in $H^2(\Gamma_K, N) = 0$, so there is a map $\varepsilon \colon \Gamma_K \to N$ such that $\Delta_c(\bar\rho) = \partial\varepsilon$, and the question is to determine whether we can choose~$\varepsilon$ such that it factors through~$\Gamma_K/\Gamma_K^v$.
  In other words, we have the following formula for the minimal lift height:
  \[
    \minlift {\bar\rho}
    =
    \max\leftl(
      \lastjump \bar\rho, \;
      \min
      \suchthat{
        \lastjump \varepsilon
      }{
        \varepsilon \colon \Gamma_K \to N \\
        \partial \varepsilon = \Delta_c(\bar\rho)
      }
    \rightr).
  \]
  For example, if~$G$ is the Heisenberg group~$H_3(\F_p)$, fitting in a short exact sequence $1 \to \F_p \to H_3(\F_p) \to \F_p^2 \to 1$, then the map $\Delta \colon H^1(\Gamma_K/\Gamma_K^v, \F_p^2) \to H^2(\Gamma_K/\Gamma_K^v, \F_p)$, via the identification $H^1(\Gamma_K/\Gamma_K^v, \F_p^2) \simeq H^1(\Gamma_K/\Gamma_K^v, \F_p)^2$, is given by the cup-product, so studying the minimal lift height for the surjection~$H_3(\F_p) \twoheadrightarrow \F_p^2$ amounts to studying the vanishing of cup-products in the cohomology of~$\Gamma_K/\Gamma_K^v$.
  Similarly, for the group~$U_4(\F_p)$ of unipotent $4\times 4$ upper-triangular matrices over~$\F_p$ (of nilpotency class~$3$), studying the minimal lift height for the surjection~$U_4(\F_p) \twoheadrightarrow \F_p^3$ (reading the coefficients just above the diagonal) amounts to studying the vanishing of triple Massey products in the cohomology of~$\Gamma_K/\Gamma_K^v$.
\end{remark}

\begin{remark}
  When $p \geq 3$, we can deduce explicit equations characterizing $\minlift \bar\rho$ from \cite[Theorem~3.20]{wildcount}.
  We will not use these characterizations in this paper.
\end{remark}

\subsection{Commutators}
\label{subsn:commutators}

The commutator map $G^2 \to G$, $(x,y) \mapsto xyx^{-1}y^{-1}$ takes values in~$N$ because~$Q$ is abelian, and factors through the surjection~$\pi \colon G^2 \twoheadrightarrow Q^2$ as~$N \subseteq Z(G)$, so we see it as a map $[-,-] \colon Q^2 \to N$.
This map is alternating ($[x,x]=0$ for any~$x \in Q$) and $\Z_p$-bilinear, i.e., biadditive.
For instance, if~$x,y,z \in G$, then
\[
  [xy,z]
  =
  xyzy^{-1}x^{-1}z^{-1}
  =
  x\underbrace{[y,z]}_{\mathclap{\in Z(G)}}zx^{-1}z^{-1}
  =
  xzx^{-1}z^{-1}[y,z]
  =
  [x,z][y,z],
\]
and additivity on the right is analogous.
Thus, it is possible to bilinearly extend the map~$[-,-]$ into a continuous $W(K^\alg)$-bilinear alternating map $[-,-] \colon Q_{K^\sep}^2 \to N_{K^\sep}$.

\subsection{Non-abelian twisting.}
\label{subsn:proof-minembed-inertial}

Consider two homomorphisms~$\rho, \delta \in \Hom(\Gamma_K, G)$.
Let~$\bar\rho := \pi\circ\rho$ and $\bar\delta := \pi\circ\delta$ be the induced homomorphisms $\Gamma_K \to Q$.
Since~$Q$ is abelian, the sum~$\bar\delta + \bar\rho$ is a group homomorphism, and an obvious continuous map~$\Gamma_K \to G$ lifting~$\bar\delta + \bar\rho$ is the pointwise product~$\delta \cdot \rho$, which is generally not a homomorphism.
However, by \Cref{lem:embpb-solvable}, we know that there is a continuous homomorphism $\rho' \in \Hom(\Gamma_K, G)$ lifting $\bar\delta + \bar\rho$, i.e., there is a continuous map~$\varepsilon \colon \Gamma_K \to N$ such that $\varepsilon \cdot \delta \cdot \rho$ is a homomorphism (take $\varepsilon := \rho' \cdot \rho^{-1} \cdot \delta^{-1}$).

\begin{lemma}
  \label{lem:epsilon-minembed}
  Let~$\varepsilon \colon \Gamma_K \to N$ be a continuous map.
  Then, the map~$\varepsilon \cdot \delta \cdot \rho$ is a homomorphism if and only if the equality $\partial \varepsilon(\tau_1, \tau_2) = [\bar\delta(\tau_2), \bar\rho(\tau_1)]$ holds in $N$ for all $\tau_1,\tau_2\in\Gamma_K$.
  In particular, this condition on~$\varepsilon$ depends only on the reductions~$\bar\delta$ and~$\bar\rho$.
\end{lemma}

\begin{proof}
  This follows from the following computation, for any~$\tau_1, \tau_2 \in \Gamma_K$:
  \begin{align*}
    (\varepsilon \cdot \delta \cdot \rho)(\tau_1 \tau_2)
    &=
    \varepsilon(\tau_1 \tau_2)
    \delta(\tau_1 \tau_2)
    \rho(\tau_1 \tau_2)
    \\
    &=
    \varepsilon(\tau_1)
    \varepsilon(\tau_2)
    \partial\varepsilon(\tau_1, \tau_2)^{-1}
    \delta(\tau_1)
    \delta(\tau_2)
    \rho(\tau_1)
    \rho(\tau_2)
    \\
    &=
    \underbrace{\varepsilon(\tau_1)}_{{\in Z(G)}}\,
    \underbrace{\varepsilon(\tau_2)}_{{\in Z(G)}}\,
    \underbrace{\partial\varepsilon(\tau_1, \tau_2)^{-1}}_{{\in Z(G)}}\,
    \delta(\tau_1)
    \underbrace{[\bar\delta(\tau_2),\bar\rho(\tau_1)]}_{{\in Z(G)}}
    \rho(\tau_1)
    \delta(\tau_2)
    \rho(\tau_2)
    \\
    &=
    \partial\varepsilon(\tau_1, \tau_2)^{-1}
    [\bar\delta(\tau_2),\bar\rho(\tau_1)]
    \cdot
    (\varepsilon \cdot \delta \cdot \rho)(\tau_1)
    \cdot
    (\varepsilon \cdot \delta \cdot \rho)(\tau_2).
  \qedhere
  \end{align*}
\end{proof}

By Artin--Schreier--Witt theory (cf.~\Cref{subsn:asw}), we fix elements~$g_{\bar\rho} ,g_{\bar\delta} \in Q_{K^\sep}$ such that $\bar\rho(\tau) = \tau(g_{\bar\rho}) - g_{\bar\rho}$ and $\bar\delta(\tau) = \tau(g_{\bar\delta}) - g_{\bar\delta}$ for all $\tau\in\Gamma_K$.
The elements~$m_{\bar\rho} := \wp(g_{\bar\rho})$ and~$m_{\bar\delta} := \wp(g_{\bar\delta})$ both belong to~$Q \otimes \cD^0 \subseteq Q_K$.

\begin{lemma}
    \label{lem:explicit-epsilon}
    Let $m_\varepsilon := [m_{\bar\rho}, g_{\bar\delta}] \in N_{K^\sep}$, let~$g_\varepsilon \in N_{K^\sep}$ be such that~$\wp(g_\varepsilon) = m_\varepsilon$, and define the continuous map~$\varepsilon \colon \Gamma_K \to N_{K^\sep}$ by
    \[
      \varepsilon(\tau) := \tau(g_\varepsilon) - g_\varepsilon + [\bar\delta(\tau), g_{\bar\rho}].
    \]
    Then:
    \begin{enumerate}[label=(\roman*)]
      \item
        \label{eps-in-N}
        For any~$\tau \in \Gamma_K$, we have~$\varepsilon(\tau) \in N$.
      \item
        \label{partial-eps}
        For any~$\tau_1, \tau_2 \in \Gamma_K$, we have $\partial \varepsilon(\tau_1, \tau_2) = [\bar\delta(\tau_2), \bar\rho(\tau_1)]$.
      \item
        \label{lastjump-eps}
        If~$\bar\delta \in \Hom_\ur(\Gamma_K, Q)$, then $\lastjump \varepsilon \leq \lastjump \bar\rho$.
    \end{enumerate}
\end{lemma}

\begin{proof}\
  \begin{enumerate}[label=(\roman*)]
  \item
  To prove~$\varepsilon(\tau) \in N = \ker \wp$, we show that $\wp(\tau(g_\varepsilon) - g_\varepsilon) = \wp([g_{\bar\rho}, \bar\delta(\tau)])$:
  \begin{align*}
      \wp(\tau(g_\varepsilon) - g_\varepsilon)
      & =
      \tau(m_\varepsilon) - m_\varepsilon
      \tag*{as~$\wp(g_\varepsilon) = m_\varepsilon$ and $\sigma$ commutes with $\tau$}
      \\
      & =
      [m_{\bar\rho}, \tau(g_{\bar\delta})] - [m_{\bar\rho}, g_{\bar\delta}]
      \tag*{as~$m_\varepsilon = [m_{\bar\rho}, g_{\bar\delta}]$ and $m_{\bar\rho} \in Q_K$}
      \\
      & =
      [\wp(g_{\bar\rho}), \bar\delta(\tau)]
      \tag*{by bilinearity, using~$m_{\bar\rho} = \wp(g_{\bar\rho})$ and $\tau(g_{\bar\delta}) - g_{\bar\delta} = \bar\delta(\tau)$}
      \\
      &=
      [\sigma(g_{\bar\rho}), \bar\delta(\tau)]
      -
      [g_{\bar\rho}, \bar\delta(\tau)]
      \tag*{by bilinearity}
      \\
      &=
      \sigma([g_{\bar\rho}, \bar\delta(\tau)])
      -
      [g_{\bar\rho}, \bar\delta(\tau)]
      \tag*{because $\bar\delta(\tau) \in Q$}
      \\
      &=
      \wp([g_{\bar\rho}, \bar\delta(\tau)]).
  \end{align*}

  \item
  For any~$\tau_1, \tau_2 \in \Gamma_K$, we have
  \begin{align*}
    \partial \varepsilon(\tau_1, \tau_2)
    & =
    \varepsilon(\tau_1) + \varepsilon(\tau_2) - \varepsilon(\tau_1 \tau_2)
    \\
    & =
    \varepsilon(\tau_1) + \tau_1(\varepsilon(\tau_2)) - \varepsilon(\tau_1 \tau_2)
    \tag*{as~$\varepsilon(\tau_2) \in N \subseteq N_K$}
    \\
    & =
    \underbrace{
      (\tau_1(g_\varepsilon) - g_\varepsilon)
      + \tau_1(\tau_2(g_\varepsilon) - g_\varepsilon)
      - (\tau_1\tau_2(g_\varepsilon) - g_\varepsilon)
    }_{= 0}
    \\
    & \phantom{={}}
    + [\bar\delta(\tau_1), g_{\bar\rho}]
    + [\bar\delta(\tau_2), \tau_1(g_{\bar\rho})]
    - [\bar\delta(\tau_1\tau_2), g_{\bar\rho}]
    \tag*{as~$\bar\delta(\tau_2) \in Q \subseteq Q_K$}
    \\
    & =
    [\bar\delta(\tau_2), \tau_1(g_{\bar\rho})]
    - [\bar\delta(\tau_2), g_{\bar\rho}]
    \tag*{by bilinearity, as $\bar\delta(\tau_1\tau_2) = \bar\delta(\tau_1) + \bar\delta(\tau_2)$}
    \\
    & =
    [\bar\delta(\tau_2), \bar\rho(\tau_1)]
    \tag*{
      by bilinearity, as $\bar\rho(\tau_1) = \tau_1(g_{\bar\rho})-g_{\bar\rho}$.
    }
  \end{align*}

  \item
  Assume that~$\bar\delta \in \Hom_\ur(\Gamma_K, Q)$.
  By \Cref{prop:local-fundom}, we have~$m_{\bar\delta} \in Q_{\F_q} = Q \otimes W(\F_q) \subseteq Q \otimes W(\bar\F_p)$, so~$g_{\bar\delta} \in Q \otimes W(\bar\F_p)$.
  ($\wp$ is surjective as an endomorphism of~$Q \otimes W(\bar\F_p)$ and has kernel $Q\subset Q\otimes W(\bar\F_p)$.)
  Our strategy to prove that $\lastjump\varepsilon\leq\lastjump\bar\rho$ is as follows: It suffices to compute $\lastjump\varepsilon$ in some unramified extension of $K$, hence we can essentially assume that $\bar\delta=0$, and then $\varepsilon$ is simply the homomorphism $\tau\mapsto\tau(g_\varepsilon)-g_\varepsilon$ associated to $m_\varepsilon$ via Artin--Schreier theory, whose last jump is computed using \Cref{prop:lj-asw}.
  To make this strategy precise, let~$q'$ be a power of~$q$ such that $g_{\bar\delta} \in Q \otimes W(\F_{q'})$, and let~$K' := \F_{q'}\ppar{T}$.
  Since~$K'|K$ is unramified, for any $v > -1$, we have $\Gamma_{K'}^v = \Gamma_K^v$.
  Write~$m_{\bar\rho} \in Q \otimes \cD^0$ as
  \[
    m_{\bar\rho}
    =
    \sum_{
      n \in \Onotp
    }
      m_{\bar\rho,n} [T]^{-n}
    \qquad
    \textnormal{with }
    \begin{cases}
      m_{\bar\rho,n} \in Q \otimes W(\F_q) & \textnormal{for all } n \in \Onotp \\
      m_{\bar\rho,n} = 0 & \textnormal{for almost all } n
    \end{cases}
  \]
  and let $v := \lastjump \bar\rho$.
  By \Cref{prop:lj-asw}, we have $p^{\mu_{v+1}(n)} m_{\bar\rho,n} = 0$ for all~$n \in \notp$.
  We have
  \[
    \wp(g_\varepsilon)
    =
    m_\varepsilon
    =
    [m_{\bar\rho}, g_{\bar\delta}]
    =
    \sum_{
      n \in \Onotp
    }
      \leftl[
          m_{\bar\rho,n},
          g_{\bar\delta}
      \rightr]
      \cdot
      [T]^{-n},
  \]
  which belongs to~$Q \otimes_{\Z_p} \cD^0 \otimes_{W(\F_q)} W(\F_{q'})$.
  Moreover, for any~$n \in \notp$, we have by bilinearity:
  \[
    p^{\mu_{v+1}(n)}
    \leftl[
          m_{\bar\rho,n},
          g_{\bar\delta}
    \rightr]
    =
    \bigl[
      p^{\mu_{v+1}(n)} m_{\bar\rho,n}, \,
      g_{\bar\delta}
    \bigr]
    =
    0
  \]
  so, by \Cref{prop:lj-asw} applied over~$K'$ (noting that~$T$ is still a uniformizer of~$K'$ as~$K'|K$ is unramified), the homomorphism $\varepsilon' \in \Hom(\Gamma_{K'}, N)$, $\tau \mapsto \tau(g_\varepsilon) - g_\varepsilon$ satisfies $\lastjump \varepsilon' \leq v$.
  Since $\bar\delta \in \Hom_\ur(\Gamma_K, Q)$, for any~$\tau$ in the inertia subgroup $\Gamma_K^0 = \Gamma_{K'}^0$, we have $\bar\delta(\tau)=0$ and hence $\varepsilon(\tau) = \tau(g_\varepsilon) - g_\varepsilon = \varepsilon'(\tau)$ by definition.
  Since the last jump only depends on the restriction of~$\varepsilon$ to inertia, we have $\lastjump \varepsilon = \lastjump \varepsilon' \leq v = \lastjump \bar\rho$.
  \qedhere
  \end{enumerate}
\end{proof}

Finally, we prove \Cref{thm:minembed-inertial}:

\begin{proof}[Proof of \Cref{thm:minembed-inertial}]
  For $i \in \{1,2\}$, pick a minimal lift $\rho_i \in \Hom(\Gamma_K, G)$ of~$\bar\rho_i$.
  All lifts of~$\bar\rho_i$ are twists of~$\rho_i$ by elements $\delta \in \Hom(\Gamma_K,N)$, and by \iref{lem:prop-min-lift}{lj-twists-min} we have $\lastjump (\delta \cdot \rho_i) = \max(\lastjump \delta, \, \lastjump \rho_i)$.
  Therefore, for any $v \geq 0$, we have
  \begin{align}
    \label{eqn:bij-minembed}
    &\cardsuchthat{
      \rho \in \Hom(\Gamma_K, G)
    }{
      \pi\circ\rho = \bar\rho_i \\
      \lastjump \rho = v
    }
    \notag
    \\
    &{}=
    \begin{cases}
      0 &
      \textnormal{if } v < \minlift \bar\rho_i,
      \\
      \cardsuchthat{
        \delta \in \Hom(\Gamma_K, N)
      }{
        \lastjump \delta \leq v
      } &
      \textnormal{if } v = \minlift \bar\rho_i,
      \\[.7em]
      \cardsuchthat{
        \delta \in \Hom(\Gamma_K, N)
      }{
        \lastjump \delta = v
      } &
      \textnormal{if } v > \minlift \bar\rho_i.
    \end{cases}
  \end{align}
  The right-hand side depends on~$\bar\rho_i$ only through~$\minlift \bar\rho_i$, so it suffices to show that $\minlift {\bar\rho_2} = \minlift {\bar\rho_1}$.
  By symmetry, proving the inequality~$\leq$ suffices: we must construct a lift~$\rho_2'$ of $\bar\rho_2$ with $\lastjump \rho_2' \leq \lastjump \rho_1$.
  Let $\bar\delta := \bar\rho_2 - \bar\rho_1$.
  By assumption, $\bar\delta$ lies in $\Hom_\ur(\Gamma_K, Q)$, and thus has an unramified lift $\delta \in \Hom_\ur(\Gamma_K, G)$ by \Cref{prop:trivial-bound}.
  Let $\varepsilon:\Gamma_K\to N$ be the continuous map defined in \Cref{lem:explicit-epsilon} (with $\rho=\rho_1$, $\delta=\delta$).
  By \Cref{lem:epsilon-minembed} and \iref{lem:explicit-epsilon}{partial-eps}, the pointwise product $\rho_2' := \varepsilon\cdot\delta\cdot\rho_1$ is a group homomorphism (and thus a lift of~$\bar\rho_2 = \bar\delta+\bar\rho_1$).
  By \Cref{eqn:ultineq} and \iref{lem:explicit-epsilon}{lastjump-eps} (with \Cref{prop:trivial-bound}), we have $\lastjump \rho_2' \leq \max(\lastjump\varepsilon,\lastjump\delta,\lastjump\rho_1) = \lastjump(\rho_1)$.
\end{proof}

Together with \Cref{thm:abstract-locglob}, \Cref{thm:minembed-inertial} directly implies \Cref{thm:intro-local-global}.

\section{An application: counting $D_4$-extensions of~$\F_q(T)$ in characteristic~$2$}
\label{sn:counting-d4}

We consider the dihedral group~$D_4$, of order~$8$ and nilpotency class~$2$, which is also the Heisenberg group $\mat{1&\F_2&\F_2\\&1&\F_2\\&&1}$.
Its center is $Z(D_4) = \mat{1&0&\F_2\\&1&0\\&&1} \simeq \F_2$, and the quotient~$D_4/Z(D_4)$ is isomorphic to the abelian group~$\F_2^2$.
(The surjection $\pi \colon D_4 \twoheadrightarrow \F_2^2$ is given by reading the two coefficients directly above the diagonal.)

In this section, we describe the distribution of last jumps of elements of $\Hom(\Gamma_K, D_4)$ when~$K$ is either a local or a rational global function field of characteristic~$2$.
(For $p \neq 2$, we have already dealt with Heisenberg $p$-extensions in characteristic~$p$ in \cite[Theorem~1.3]{wildcount}.)

\subsection{Local distribution}
\label{subsn:local-d4}

Fix a finite field~$\F_q$ of characteristic~$2$.
Let~$K := \F_q\ppar{T}$.
For any $x \in K$, let $w(x) := \max(0, -v_T(x))$ where $v_T$ is the $T$-adic valuation.
By Artin--Schreier theory (see \Cref{rk:elemab}), the elements of $\Hom(\Gamma_K,\F_2)$ are parametrized by elements of the set $\cD^0 := \bigoplus_{n \in \Onottwo} \F_q T^{-n} \subset \F_q\ppar{T}$ (which was called $\F_2\otimes\cD^0$ in \Cref{subsn:asw}); moreover, the last
jump of the element of $\Hom(\Gamma_K,\F_2)$ corresponding to $x\in\cD^0$ is $w(x)$, two elements $x, x' \in \cD^0$ correspond to the same element of $\Hom(\Gamma_K,\F_2)$ if and only if $x - x' \in \wp(\F_q)$, and the unramified homomorphisms correspond to the elements of $\F_q\subseteq\cD^0$.

\begin{proposition}
  \label{prop:minembed-D4}
  Let~$a,c \in \cD^0$.
  Let $\bar\rho\in\Hom(\Gamma_K,\F_2^2)$ be the homomorphism corresponding to $(a,c)$ via Artin--Schreier theory, i.e., the map $\tau \mapsto (\tau(\alpha) - \alpha, \tau(\gamma) - \gamma)$ for any~$\alpha \in \wp^{-1}(a)$ and~$\gamma \in \wp^{-1}(c)$.
  Then, the minimal lift height of $\bar\rho$ with respect to $\pi\colon D_4 \twoheadrightarrow \F_2^2$ is
  $
    \minlift \bar\rho = w(a) + w(c)
  $.
\end{proposition}

\begin{proof}
  By \Cref{rk:elemab}, we have $\lastjump \bar\rho = \max(w(a), w(c))$.
  Write $a = \sum_{n \in \Onottwo} a_n T^{-n}$, and $c = \sum_{n \in \Onottwo} c_n T^{-n}$.
  By \Cref{thm:minembed-inertial}, $\minlift \bar\rho$ is unchanged if we apply unramified twists to~$\bar\rho$, so we assume that $a_0 = c_0 = 0$.
  We now consider the possible inertia subgroups~$\bar\rho(\Gamma_K^0) \subseteq \F_p^2$.

  If $\bar\rho(\Gamma_K^0) \subseteq 0\oplus\F_2$, then $w(a)=0$, so $a_0=0$ implies $a=0$, so we in fact have $\bar\rho(\Gamma_K)\subseteq0\oplus\F_2$.
  Consider the subgroup~$H := \mat{1&0&0\\&1&\F_2\\&&1}$ of~$D_4$.
  In restriction to~$H$, the map~$\pi$ is an isomorphism $H \simto 0 \oplus \F_2$.
  Let~$i \colon 0 \oplus \F_2 \simto H$ be its inverse.
  Then, $i \circ \bar\rho$ is a lift of~$\bar\rho$ with same last jump as~$\bar\rho$, namely $\max(w(a), w(c)) = w(c) = w(a)+w(c)$.
  Together with the inequality $\minlift\bar\rho\geq\lastjump\bar\rho$ of \Cref{prop:trivial-bound}, this implies the claim.
  The case $\bar\rho(\Gamma_K^0)\subseteq\F_2\oplus0$ is treated similarly.

  If $\bar\rho(\Gamma_K^0)$ is the diagonal $\Delta \simeq \F_2$ in~$\F_2^2$, then $w(a-c)=0$, so $a_0=c_0=0$ implies $a = c$, so~$\bar\rho(\Gamma_K)$ is also the diagonal~$\Delta$.
  Any lift~$\rho$ of~$\bar\rho$ must satisfy $\rho(\Gamma_K) \subseteq \pi^{-1}(\Delta) \simeq \Z/4\Z$.
  Therefore, $\minlift\bar\rho$ is also the minimal lift of~$\bar\rho$ relative to the surjection $\Z/4\Z \twoheadrightarrow \F_2\simeq\Delta$, and using \Cref{prop:lj-asw} it is elementary to deduce that $\minlift\bar\rho=2\lastjump\bar\rho = w(a)+w(c)$ (a minimal lift is given by the $\Z/4\Z$-extension associated to $\sum [a_n] [T]^{-n}$, where~$[a_n]$ is the Teichmüller lift of~$a_n=c_n$ in $W(K^\perf)$).

  We now focus on the remaining case $\bar\rho(\Gamma_K^0) = \F_2^2$.
  As there are no proper subgroups of~$D_4$ whose image under~$\pi$ is~$\F_2^2$, this implies $\rho(\Gamma_K^0)=D_4$ for all lifts~$\rho$ of~$\bar\rho$, so any lift is surjective and totally ramified.
  By \cite[Example~3.8.1]{imai-wild-ramification-groups} (see also \cite[Subsections~2.2 and~3.1]{wildcount}), lifts~$\rho \in \Hom(\Gamma_K, D_4)$ of $\bar\rho$ are parametrized by matrices $M := \mat{1&a&b\\&1&c\\&&1} \in H_3(K)$ with $b \in \cD^0$.
  Explicitly, letting $g \in H_3(K^\sep)$ be a matrix such that $\sigma(g)g^{-1} = M$, we have $\rho(\tau) = g^{-1} \tau(g)$.

  By \Cref{rk:elemab}, $w(a)$ and~$w(c)$ are both odd, so $w(a)+w(c)$ is even.
  By \iref{lem:prop-min-lift}{not-integral-is-minimal}, it thus suffices to prove that some lift has last jump~$w(a) + w(c)$, and it will necessarily be minimal.
  %
  For any $x = \sum_{n \in \Onottwo} x_n T^{-n} \in \cD^0$, let $x' := T \frac{\d x}{\d T} = - \sum_{n \in \nottwo} n x_n T^{-n}$.
  By \cite[Example~3.8.1]{imai-wild-ramification-groups}, we have a formula for~$\lastjump\rho$ in terms of the parametrization above:
  \begin{equation}
    \label{eqn:lj-d4}
    \lastjump \rho
    =
    \max\leftl(
      w\bigl(
        b' - a c'
      \bigr), \;
      \frac{w(a)}2+w(c), \;
      \frac{w(c)}2+w(a)
    \rightr).
  \end{equation}
  Take~$b=0$.
  Then, we have $w(b'-ac') = w(ac') = w(a) + w(c') = w(a) + w(c)$, and the two other terms are $\leq w(a)+w(c)$, concluding the proof.
\end{proof}

The formula for~$\minlift \bar\rho$ given in \Cref{prop:minembed-D4} is invariant under adding elements of~$\F_q$ to~$a$ or~$c$: this is a concrete manifestation of \Cref{thm:minembed-inertial}.
Another consequence of \Cref{prop:minembed-D4} (see also \iref{lem:prop-min-lift}{not-integral-is-minimal}) is that the last jump of any~$\rho \in \Hom(\Gamma_K, D_4)$ is an integer.
(This was already observed in \cite[Remark~1.7]{elder}.)

\begin{corollary}
  \label{cor:local-distrib-D4}
  For each integer~$v \geq 0$, we have
  \[
    \frac14
    \cardsuchthat{
      \bar\rho \in \Hom\bigl(\Gamma_K, \F_2^2\bigr)
    }{
      \minlift \bar\rho = v
    }
    =
    \begin{cases}
      1 &
      \textnormal{if } v = 0 \\
      2 q^{\frac{v-1}2} (q-1) &
      \textnormal{if } v \in 2\N - 1 \\
      \frac v2 q^{\frac v 2 - 1}(q-1)^2 &
      \textnormal{if } v \in 2\N.
    \end{cases}
  \]
\end{corollary}

\begin{proof}
  By \Cref{eq:unram-ext}, dividing by $4 = \card{\F_2^2} = \card{\Hom_\ur(\Gamma_K, \F_2^2)}$ amounts to counting the inertial types of~$\bar\rho$,
  (for which~$\minlift$ makes sense by \Cref{thm:minembed-inertial}, or directly by \Cref{prop:minembed-D4})
  i.e., to counting pairs of elements $a,c$ of $\cD^0/\F_q$.
  In other words, we only need to count choices of $a_n$ and $c_n$ for $n\in\nottwo$, and we ignore the number of choices for $a_0$ and $c_0$ (cf.~\Cref{prop:local-fundom}).
  By \Cref{prop:minembed-D4}, we have $\minlift \bar\rho = v$ if and only if $w(a) + w(c) = v$.

  For~$v = 0$, we get the unique unramified inertial type.

  Assume that~$v$ is odd.
  Since~$w(a)$ and~$w(c)$ are either odd or zero, the only two possibilities are that $w(a)=v$ and $w(c)=0$, or $w(a)=0$ and $w(c) = v$.
  The number of~$x \in \cD^0/\F_q$ such that $w(x) = v$ is $q^{\frac{v-1}2} (q-1)$, giving the result in that case.

  Assume now that~$v > 0$ is even.
  The pair $(w(a), w(c))$ must be of the form $(i, v-i)$ for some odd~$i$ with $1 \leq i \leq v-1$.
  For each such $i$ (of which there are~$\frac v2$), the number of corresponding choices for $(a,c) \in (\cD^0/\F_q)^2$ is
  $
    q^{\frac{i-1}2}(q-1)\cdot q^{\frac{v-i-1}2} (q-1)
    =
    q^{\frac v2 - 1} (q-1)^2
  $.
\end{proof}

By combining the above corollary with the following lemma, one obtains exact counts for the number of $D_4$-extensions of~$K = \F_q\ppar{T}$ with a given last jump.

\begin{lemma}
  \label{lem:d4-from-f22}
  For each integer $v\geq0$, we have
  \[
    \frac18
    \cardsuchthat{
      \rho \in \Hom(\Gamma_K, D_4)
    }{
      \lastjump \rho \leq v
    }
    = 
    q^{\ceil{v/2}}
    \cdot
    \sum_{v' = 0}^v
    \frac14
    \cardsuchthat{
      \bar\rho \in \Hom(\Gamma_K, \F_2^2)
    }{
      \minlift \bar\rho = v'
    }.
  \]
\end{lemma}

\begin{proof}
  Since $Z(D_4) \simeq \F_2$, it follows from \Cref{rk:elemab} that
  \begin{equation}
    \label{eqn:count-C2-ext}
    \frac12 \cardsuchthat{
      \delta \in \Hom(\Gamma_K, Z(D_4))
    }{
      \lastjump \delta \leq v
    }
    =
    q^{\cardsuchthat{n \in \nottwo}{n \leq v}}
    =
    q^{\ceil{v/2}}.
  \end{equation}
  We then have:
  \begin{align*}
    &\frac18\,
    \cardsuchthat{
      \rho \in \Hom(\Gamma_K, D_4)
    }{
      \lastjump \rho \leq v
    }
    \notag
    \\={}&
    \frac18\,
    \sum_{
      \bar\rho \in \Hom(\Gamma_K, \F_2^2)
    }
      \cardsuchthat{
        \rho \in \Hom(\Gamma_K, D_4)
      }{
        \pi\circ\rho = \bar\rho \\
        \lastjump \rho \leq v
      }
    \notag
    \\={}&
    \frac18\,\sum_{\substack{
      \bar\rho \in \Hom(\Gamma_K, \F_2^2) \\
      \minlift {\bar\rho} \leq v
    }}
      \cardsuchthat{
        \delta \in \Hom(\Gamma_K, Z(D_4))
      }{
        \lastjump \delta \leq v
      }
    \tag*{by \Cref{eqn:bij-minembed}}
    \\={}&
    q^{\ceil{v/2}}
    \cdot
    \frac14
    \cardsuchthat{
      \bar\rho \in \Hom(\Gamma_K, \F_2^2)
    }{
      \minlift \bar\rho \leq v
    }
    &
    \tag*{by \Cref{eqn:count-C2-ext}}
    \\={}&
    q^{\ceil{v/2}}
    \cdot
    \sum_{v' = 0}^v
    \frac14
    \cardsuchthat{
      \bar\rho \in \Hom(\Gamma_K, \F_2^2)
    }{
      \minlift \bar\rho = v'
    }.
    \tag*\qedhere
  \end{align*}
\end{proof}

\subsection{Global counting}

Now, fix a rational global function field~$K := \F_q(T)$ of characteristic~$2$.
For any prime~$\fp \in \cP_K$, let $\kappa_\fp := \F_{q^{\deg \fp}}$ be the residue field of~$\cO_\fp$.
For each prime~$\fp \in \cP_K$ and for all~$v \in \N$, we define the number
\[
  a_{\fp, v}
  :=
  \frac18
  \cardsuchthat{
    \rho \in \Hom(\Gamma_{K_\fp}, D_4)
  }{
    \lastjump \rho = v
  }.
\]

\begin{lemma}
  We have the following estimates:
  \[
    a_{\fp, 0} = 1,
    \qquad\qquad
    a_{\fp, 1} = 2 |\kappa_\fp|^2 + O(|\kappa_\fp|),
    \qquad\qquad
    a_{\fp, v} = O(v|\kappa_\fp|^{v+1}).
  \]
\end{lemma}

\begin{proof}
  First, we indeed have
  \[
    a_{\fp,0} = \frac18 \card{\Hom_\ur(\Gamma_{K_\fp}, D_4)} = 1
    \tag*{by~\Cref{eq:unram-ext}.}
  \]
  For larger values of~$v$, we use \Cref{lem:d4-from-f22} and \Cref{cor:local-distrib-D4}.
  The size of the residue field at~$\fp$ (the ``$q$'' in the local estimates) is $\card {\kappa_\fp}$.
  We have
  \begin{align*}
    a_{\fp, 1}
    &=
    \frac18\cardsuchthat{
      \rho \in \Hom(\Gamma_{K_\fp}, D_4)
    }{
      \lastjump \rho \leq 1
    }
    -
    a_{\fp,0}
    \\
    &=
    \card {\kappa_\fp}
    \cdot
    \sum_{v'=0}^{1}
      \frac14
      \cardsuchthat{
        \bar\rho \in \Hom(\Gamma_{K_\fp}, \F_2^2)
      }{
        \minlift \bar\rho = v'
      }
    -
    1
    \tag*{by \Cref{lem:d4-from-f22}}
    \\
    &=
    \card {\kappa_\fp}
    \cdot
    \leftl(
      2(\card {\kappa_\fp}-1)
      +1
    \rightr)
    -
    1
    =
    2\card {\kappa_\fp}^2 + O(\card {\kappa_\fp})
    \tag*{by \Cref{cor:local-distrib-D4}.}
  \end{align*}
  Finally, for arbitrary~$v$:
  \begin{align*}
    a_{\fp, v}
    &\leq
      \frac18
      \cardsuchthat{
        \rho \in \Hom(\Gamma_{K_\fp}, D_4)
      }{
        \lastjump \rho \leq v
      }
    \\
    &=
      \card {\kappa_\fp}^{\lceil v/2\rceil}
      \sum_{v' = 0}^v
      \frac14
      \cardsuchthat{
        \bar\rho \in \Hom(\Gamma_{K_\fp}, \F_2^2)
      }{
        \minlift \bar\rho = v'
      }
    \tag*{by \Cref{lem:d4-from-f22}}
    \\
    &=
      \card {\kappa_\fp}^{\lceil v/2\rceil}
      \sum_{v' = 0}^v
        O(v'\card{\kappa_\fp}^{\lfloor v'/2\rfloor+1})
    \tag*{by \Cref{cor:local-distrib-D4}}
    \\
    &=
    O\leftl(
      \card {\kappa_\fp}^{\lceil v/2\rceil}
      v
      \card{\kappa_\fp}^{\lfloor v/2\rfloor+1}
    \rightr)
    \\
    &=
    O(
      v
      \card{\kappa_\fp}^{v+1}
    ).
    \tag*\qedhere
  \end{align*}
\end{proof}

Using \Cref{thm:intro-local-global} and \cite[Lemma~5.4]{wildcount}, we obtain the following estimate for the distribution of $D_4$-extensions of~$\F_q(T)$ ordered by their global last jump (cf.~\Cref{def:global-height}):

\begin{theorem}
  \label{thm:count-global-D4}
  There is a real number~$C > 0$ such that for integers~$X \to \infty$ we have
  \[
    \cardsuchthat{
      \rho \in \Hom(\Gamma_K, D_4)
    }{
      \sum_{\fp \in \cP_K}
        \deg \fp \cdot \lastjump\leftl(\rho\restfp\rightr)
      =
      X
    }
    =
    C q^{3 X} X  + o(q^{3X} X).
  \]
\end{theorem}

\begin{proof}
  Pick any $\varepsilon \in (0,1)$.
  Observe that we have $a_{\fp, v} = O(v|\kappa_\fp|^{v+1}) = O(|\kappa_\fp|^{(1+\varepsilon)v+1})$.
  Moreover, if $v \geq 2$, then
  \[
    \frac{(1+\varepsilon)v+1 + 1}v
    =
    1+\frac2v+\varepsilon
    \leq
    2+\varepsilon
    <
    3
    =\frac{2+1}1.
  \]
  Therefore, in the notation of \cite[Lemma~5.4]{wildcount}, we have $S = \{1\}$, $A = \frac{2+1}{1} = 3$, $B = 2$, $M = 1$, and the hypothesis ``(5.4)'' is satisfied.
  We have
  \begin{align*}
    &\frac18
    \cardsuchthat{
      \rho \in \Hom(\Gamma_K, D_4)
    }{
      \sum_{\fp \in \cP_K}
        \deg \fp \cdot \lastjump(\rho\restfp)
      =
      X
    }
    \\
    ={}&
    \sum_{\substack{
      (v_\fp) \in \prod_\fp \Z_{\geq 0} \\
      \sum_\fp \deg \fp \cdot v_\fp = X
    }}
      \frac18
      \cardsuchthat{
        \rho \in \Hom(\Gamma_K, D_4)
      }{
        \forall \fp, \lastjump(\rho\restfp)
        =
        v_\fp
      }
    \\
    ={}&
    \sum_{\substack{
      (v_\fp) \in \prod_\fp \Z_{\geq 0} \\
      \sum_\fp \deg \fp \cdot v_\fp = X
    }}
      \prod_\fp
        a_{\fp, v_\fp}
    \tag*{by \Cref{thm:intro-local-global}}
    \\
    ={}&
    C q^{3 X} X  + o(q^{3X} X)
    \tag*{by \cite[Lemma~5.4]{wildcount}}
  \end{align*}
  for some $1$-periodic function $C \colon \Q/\Z \to \R_{\geq 0}$ with $C(0) > 0$.
  Restricting to integers~$X$ (the only possibility), $C$ is thus a positive real constant.
\end{proof}

\section{On (the failure of) generalizations of \Cref*{thm:intro-local-global}}
\label{sn:generalizations}

In this section, we discuss the possibility of generalizing \Cref{thm:intro-local-global}.
In \Cref{thm:single-prime}, we show that, for extensions ramified at a single prime of degree $\notin p\N$, there is a local--global bijection valid for all finite $p$-groups, without reference to any specific inertial height.
In \Cref{prop:counterexample-discr}, we give a counterexample to the naive generalization of \Cref{thm:intro-local-global} to discriminants.
We leave for future work the study of statements in the spirit of \Cref{thm:intro-local-global} for groups of higher nilpotency class and/or over non-rational base fields.

\medskip

Before diving in, we state a version of \Cref{lem:embpb-solvable} with restricted ramification, which we use in the proofs of \Cref{thm:single-prime} and of \Cref{prop:counterexample-discr}:

\begin{lemma}
  \label{unram-embpb}
  Let~$K=\F_q(T)$ be a rational function field of characteristic~$p$.
  Consider a short exact sequence $1 \to N \to G \overset\pi\to Q \to 1$, and assume that~$N \subseteq Z(G)$ is a finite abelian $p$-group.
  Let~$S$ be a set of primes of~$K$, and let $\bar\rho \in \Hom(\Gamma_K, Q)$ be unramified outside~$S$.
  Then, there exists a lift $\rho \in \Hom(\Gamma_K, G)$ of~$\bar\rho$ which is unramified outside~$S$, and all such lifts are exactly the homomorphisms of the form $\delta\cdot\rho$ where $\delta \in \Hom(\Gamma_K, N)$ is unramified outside~$S$.
\end{lemma}

\begin{proof}
  By \Cref{lem:embpb-solvable}, we can pick a lift $\rho_0\in\Hom(\Gamma_K,G)$ of~$\bar\rho$.
  The restriction $\bar\rho|_{\Gamma_{K_\fp}} \in\Hom(\Gamma_{K_\fp},Q)$ is by assumption unramified for every $\fp\notin S$, so by \Cref{eq:unram-lift} it has an unramified lift to $G$, which must be the twist $\varepsilon_\fp\cdot\rho_0|_{\Gamma_{K_\fp}} \in \Hom_\ur(\Gamma_{K_\fp},G)$ of $\rho_0|_{\Gamma_{K_\fp}}$ by some $\varepsilon_\fp\in\Hom(\Gamma_{K_\fp},N)$.
  By the local-global principle for abelian extensions (\Cref{lem:cft}), there is some $\varepsilon\in\Hom(\Gamma_K,N)$ such that $\varepsilon|_{\Gamma_{K_\fp}}\cdot\varepsilon_\fp^{-1}\in\Hom(\Gamma_{K_\fp},N)$ is unramified for all $\fp\notin S$.
  Then, the lift $\rho := \varepsilon\cdot\rho_0 \in \Hom(\Gamma_K, G)$ of $\bar\rho$ is unramified outside $S$.
  
  By \Cref{lem:embpb-solvable}, all other lifts of $\bar\rho$ are of the form $\delta\cdot\rho$ with $\delta\in\Hom(\Gamma_K, N)$. Clearly, such a lift is unramified outside $S$ if and only if $\delta$ is.
\end{proof}

\subsection{Extensions ramified at a single prime of degree not divisible by~$p$}

Let~$K = \F_q(T)$, and let~$\fp$ be a prime of~$K$ of degree not divisible by~$p$.
In \Cref{thm:single-prime}, we establish a local--global principle for finite $p$-extensions of~$K$ unramified outside~$\fp$ of arbitrary nilpotency class.
This means that the case ``$v_\fq = 0$ whenever $\fq\neq\fp$'' of \Cref{thm:intro-local-global} is much easier to prove and valid in much more generality.
The reason is that in this case we do not have to deal with the indeterminacy in \Cref{lem:cft}: by adding a global unramified homomorphism, we can always prescribe the unramified behaviour at~$\fp$ of any global extension exactly.

\begin{theorem}
  \label{thm:single-prime}
  Let~$G$ be any finite $p$-group.
  Then, the restriction map
  \[
    \suchthat{\rho \in \Hom(\Gamma_K, G)}{\rho \textnormal{ unramified outside } \fp}
    \longrightarrow
    \Hom(\Gamma_{K_\fp}, G)
  \]
  is a bijection.
\end{theorem}

\begin{proof}
  We prove this by induction on $\card G$.
  The case $\card G = 1$ is trivial, so we assume that $G \neq 1$ and that the result holds for $p$-groups of size $< \card G$.
  The non-trivial $p$-group~$G$ has a non-trivial center, so pick a central subgroup $N \subseteq Z(G)$ with $N \simeq \F_p$.
  By the induction hypothesis, the result holds for $G/N$, so we can reason in a given fiber: we fix a $\bar\rho \in \Hom(\Gamma_K, G/N)$ unramified outside~$\fp$, and we have to show
  \begin{equation}
    \label{fiberwise-locglob}
    \suchthat{\rho \in \Hom(\Gamma_K, G)}{\rho\bmod N=\bar\rho\\\rho\textnormal{ unramified outside }\fp}
    \simeq
    \suchthat{\rho \in \Hom(\Gamma_{K_\fp}, G)}{\rho\bmod N=\bar\rho\restfp}.
  \end{equation}
  Use \Cref{unram-embpb} to pick a lift $\rho_0 \in \Hom(\Gamma_K, G)$ of~$\bar\rho$ which is unramified outside~$\fp$.
  The lifts of~$\bar\rho$ that are unramified outside~$\fp$ are exactly the ~$\delta\cdot\rho_0$ with $\delta \in \Hom(\Gamma_K, N)$ unramified outside~$\fp$.
  Similarly, the lifts of~$\bar\rho\restfp$ are of the form~$\delta\cdot\rho_0\restfp$ with $\delta \in \Hom(\Gamma_{K_\fp}, N)$.
  Hence, \Cref{fiberwise-locglob} reduces to proving
  \[
    \suchthat{\delta \in \Hom(\Gamma_K, N)}{\delta\textnormal{ unramified outside }\fp}
    \simeq
    \Hom(\Gamma_{K_\fp}, N).
  \]
  Recall that $N \simeq \F_p$.
  By \Cref{lem:cft}, we know this bijection modulo unramified homomorphisms.
  Hence, it is enough to prove that the restriction map $\Hom_\ur(\Gamma_K, \F_p) \simto \Hom_\ur(\Gamma_{K_\fp}, \F_p)$ is a bijection.
  We have $\Gal(K^\ur|K) \simeq \hat\Z$ and, for the same identification, $\Gal(K_\fp^\ur|K_\fp) \simeq \deg\fp \cdot \hat\Z$ (cf.~\Cref{inclusion-frobenius}).
  Recall that $\deg\fp$ is coprime to~$p$ (i.e., invertible in~$\F_p$), so indeed the restriction map $\Hom(\hat\Z, \F_p) \to \Hom(\deg\fp\cdot\hat\Z,\F_p)$ is bijective.
\end{proof}

\begin{remark}
  Another way to state \Cref{thm:single-prime} is that the injection $\Gamma_{K_\fp} \hookrightarrow \Gamma_K$ induces an isomorphism between the Galois group of the maximal $p$-extension of~$K_\fp$
  and the Galois group of the maximal $p$-extension of~$K$ unramified outside~$\fp$.
  The latter group is pro-$p$, hence pro-nilpotent, so that isomorphism can also be established by checking the surjectivity at the level of their abelianizations (cf.~\cite[Lemma~5.9]{cgt}), which follows via class field theory from the surjectivity of the map $\Z_p \to \Z_p$, $x \mapsto \deg\fp \cdot x$.
  This yields an alternative proof.
\end{remark}

\subsection{A counterexample for the discriminant exponent}

Let $K = \F_q(T)$ be a rational function field of characteristic $p\neq2$.
In \Cref{prop:counterexample-discr}, we give a counterexample to the naive generalization of \Cref{thm:intro-local-global} in which the last jump is replaced with the discriminant exponent~$\delta$, defined in~\Cref{eq:def-discr-exp}.

Observe that the local--global principle (for discriminant exponents, or for any inertial height) does hold for abelian $p$-extensions by \Cref{lem:cft}, and for $p$-extensions ramified at a single prime of degree coprime to~$p$ by \Cref{thm:single-prime}.
Therefore, in order to construct a ``minimal'' counterexample, we pick a prime~$\fp \in \cP_K$ of degree~$p$, and we consider the smallest possible non-abelian $p$-group $G := H_3(\F_p) = \left(\begin{smallmatrix}1&\F_p&\F_p\\&1&\F_p\\&&1\end{smallmatrix}\right)$.
We will witness a failure of the local--global principle for the elements of $\Hom(\Gamma_K, G)$ that are unramified outside~$\fp$.

The smallest possible local discriminant exponent at~$\fp$ leading to a ramified extension is
\[
  2p^2(p-1) = p^3 \cdot \int_{-1}^1 \leftl(1-\frac1p\rightr) \d v,
\]
for which the only possibility is that the inertia subgroup has size~$p$, and that the last jump is~$1$.
We shall see that the number of global $G$-extensions of~$K$ unramified outside~$\fp$ with that behavior at~$\fp$ is in fact always larger than the number of local $G$-extensions of~$K_\fp$ with that behavior, contradicting the analogue of \Cref{thm:intro-local-global} with last jump replaced by discriminant exponent.

The reason for this difference is that the Frobenius elements are very constrained in the local case (they must lie in the normalizer of the inertia subgroup, which has size~$p^2$ for all possible inertia subgroups of size~$p$ except~$Z(G)$), whereas in the global case the Frobenius element at~$\fp$ is automatically trivial by \Cref{inclusion-frobenius} as $\deg\fp = p$ and~$G$ has exponent~$p$, so there is no constraint and we can pick the Frobenius element freely in~$G$ (giving~$p^3$ choices).

\medskip

We let $N := Z(G) = \left(\begin{smallmatrix}1&0&\F_p\\&1&0\\&&1\end{smallmatrix}\right)$ and $Q = G/N \simeq \F_p^2$.
We denote as usual by $\pi \colon G \twoheadrightarrow Q$ the canonical surjection (reading the two coefficients just above the diagonal).

\begin{proposition}
  \label{prop:counterexample-discr}
  Let $K = \F_q(T)$ be a rational function field of characteristic $p\neq2$, let $\fp$ be a prime of degree~$p$ of~$K$, and let $G = H_3(\F_p)$.
  On the local side, we have
  \[
    \cardsuchthat{
      \rho \in \Hom(\Gamma_{K_\fp}, G)
    }{
      \delta(\rho) = 2p^2(p-1)
    }
    =
    p^3(p+2)(q^p-1)
  \]
  whereas, on the global side,
  \[
    \cardsuchthat{
      \rho \in \Hom(\Gamma_K, G)
    }{
      \delta(\rho\restfp) = 2p^2(p-1) \\
      \rho \textnormal{ unramified outside } \fp
    }
    =
    p^3(p^2+p+1)(q^p-1).
  \]
\end{proposition}

\begin{proof}
  As we have seen, the constraint on the discriminant at~$\fp$ means (both locally and globally) that $|{\rho(\Gamma_{K_\fp}^0)}| = p$, and that the last jump at~$\fp$ must be~$1$.
  
  We first count (simultaneously in the local and the global case) the number of~$\rho$ (unramified outside~$\fp$) for which $\rho(\Gamma_{K_\fp}^0) = N$, with last jump~$1$ (at~$\fp$).
  The corresponding $Q$-extension $\bar\rho := \pi\circ\rho$ must be unramified.
  Fix an unramified $Q$-extension~$\bar\rho$ (by \Cref{eq:unram-ext}, there are $\card Q = p^2$ choices for~$\bar\rho$, both locally and globally), and pick an unramified lift of~$\bar\rho$ (\Cref{prop:trivial-bound}).
  Via twisting, the number of choices for~$\rho$ in the fiber above~$\bar\rho$ is seen to equal the number of $N$-extensions with last jump~$1$ (at~$\fp$, unramified outside~$\fp$), which is $p (q^p-1)$ by \Cref{lem:count-elemab-lj1} below (with $A=B=N$).
  Therefore, both locally and globally, the count for the case $\rho(\Gamma_{K_\fp}^0) = N$ is~$p^3 (q^p-1)$.

  The other possibility is that $I := \rho(\Gamma_{K_\fp}^0)$ is some other subgroup of order~$p$ of~$G$.
  Then,~$L := \pi(I)$ is a line in $Q \simeq \F_p^2$ (it is non-trivial because $I \neq N$, and it has order $\leq \card I= p$).
  We fix such a line~$L \subseteq \F_p^2$ (there are $p+1$ choices for~$L$).
  The number of subgroups $I \subseteq G$ of order~$p$ such that $\pi(I) = L$ is~$p$ (the subgroup $\pi^{-1}(L)$ has order~$p^2$ since $\ker\pi=N$ has size~$p$, so $\pi^{-1}(L) \simeq \F_p^2$ and we are again picking a line in~$\F_p^2$, but that line must not be~$N$).
  We fix such a subgroup $I$, and we count the $G$-extensions~$\rho$ (unramified outside~$\fp$) with last jump~$1$ (at~$\fp$) for which $\rho(\Gamma_{K_\fp}^0) = I$.
  The claim will follow if we show that there are~$p^2(q^p-1)$ such~$\rho$ locally, and~$p^3(q^p-1)$ such~$\rho$ globally.
  
  Any~$\rho$ satisfying $\rho(\Gamma_{K_\fp}^0) = I$ must satisfy $\rho(\Gamma_{K_\fp}) \subseteq \pi^{-1}(L)$ (otherwise, $\pi(\rho(\Gamma_{K_\fp})) = Q$ and then $\rho(\Gamma_{K_\fp}) = G$, contradicting the fact that the inertia subgroup~$I$ is a normal subgroup).
  Therefore, counting the possible local~$\rho \in \Hom(\Gamma_{K_\fp}, G)$ actually means counting the~$\rho \in \Hom(\Gamma_{K_\fp}, \pi^{-1}(L))$ with $\rho(\Gamma_{K_\fp}^0) = I$ and $\lastjump \rho  = 1$.
  (Recall that $\pi^{-1}(L) \simeq \F_p^2$.)
  By \Cref{lem:count-elemab-lj1} below (with $A = \pi^{-1}(L)$, $B = I$), there are $p^2(q^p-1)$ such~$\rho$.

  We now count the global~$\rho \in \Hom(\Gamma_K, G)$, unramified outside~$\fp$, with last jump~$1$ at~$\fp$, satisfying $\rho(\Gamma_{K_\fp}^0) = I$.
  By \Cref{lem:count-elemab-lj1} below (with $A=Q$, $B=L$), there are~$p^2(q^p-1)$ homomorphisms $\bar\rho \in \Hom(\Gamma_K, Q)$ unramified outside~$\fp$ with $\bar\rho(\Gamma_{K_\fp}^0) = L$ and with last jump~$1$ at~$\fp$.
  Fix one such~$\bar\rho$, and use \Cref{unram-embpb} to lift~$\bar\rho$ into a $\rho_0 \in \Hom(\Gamma_K, G)$ unramified outside~$\fp$.
  Our goal is to show that there are exactly~$p$ homomorphisms $\delta \in \Hom(\Gamma_K, N)$ unramified outside~$\fp$ with $\lastjump ((\delta\cdot\rho_0)\restfp) = 1$ and $(\delta \cdot \rho_0)(\Gamma_{K_\fp}^0) = I$, i.e., that there is a unique inertial type of such~$\delta$ (we have $\card N = p$).
  By \Cref{lem:cft}, it suffices to count the inertial types of~$\delta$ locally at~$\fp$.
  The map $\bar\rho \bmod L \colon \Gamma_K \to Q/L$ is unramified (as $\bar\rho(\Gamma_{K_\fq}^0)$ is trivial for $\fq\neq\fp$ and is~$L$ for $\ell=\fp$), so factors through $\Gamma_K/\Gamma_{K^\ur} \simeq \hat\Z$.
  Similarly, the map $\bar\rho\restfp \bmod L$ factors through $\Gamma_{K_\fp}/\Gamma_{K_\fp^\ur} \simeq \deg\fp \cdot \hat\Z = p\hat\Z$ (cf.~\Cref{inclusion-frobenius})---the corresponding induced map is the restriction to~$p \hat\Z$ of the map $\hat\Z \to Q/L$ induced by~$\bar\rho \bmod L$, but~$Q/L$ has exponent~$p$, so $\bar\rho\restfp \bmod L = 0$.
  This proves that $\bar\rho(\Gamma_{K_\fp}) = L$.
  Now, $\pi$ induces an isomorphism $I \simto L$.
  Call its inverse~$s$.
  A lift~$\rho_\fp \in \Hom(\Gamma_{K_\fp}, G)$ of~$\bar\rho\restfp$ satisfies $\rho_\fp(\Gamma_{K_\fp}^0) = I$ if and only if it coincides with~$s \circ \bar\rho\restfp$ in restriction to~$\Gamma_{K_\fp}^0$---it then automatically has last jump equal to~$1$ as this is the last jump of~$\bar\rho$ at~$\fp$.
  This means that the only allowed local inertial type for $[\delta\restfp]$ is $[(s \circ \bar\rho\restfp) \cdot \rho_0^{-1}\restfp]$.
\end{proof}

\begin{lemma}
  \label{lem:count-elemab-lj1}
  Let~$A \simeq \F_p^r$, and let $B \subseteq A$ be a subgroup of~$A$ of order~$p$.
  Then
  \begin{align*}
    \cardsuchthat{\rho \in \Hom(\Gamma_K, A)}{\rho(\Gamma_{K_\fp}^0) = B \\ \rho {\textnormal{ unramified outside } \fp} \\ \lastjump \rho = 1}
    &=
    \cardsuchthat{\rho \in \Hom(\Gamma_{K_\fp}, A)}{\rho(\Gamma_{K_\fp}^0) = B \\ \lastjump (\rho\restfp) = 1}
    \\&=
    p^r (q^p - 1).
  \end{align*}
\end{lemma}

\begin{proof}
  Both locally and globally, each inertial type has size~$\card A$, and at the level of inertial types we have the local--global bijection of \Cref{lem:cft}, so it suffices to show that there are~$q^p - 1$ local inertial types.
  By \Cref{rk:elemab}, the inertial types $[\rho] \in \cI(K_\fp, A)$ with last jump~$1$ correspond to elements $m_1 \in (A \otimes \F_{q^p}) \setminus \{0\}$.
  The condition $\rho(\Gamma_{K_\fp}^0) = B$ means that~$\rho \bmod B$ is unramified, i.e., $m_1 \bmod (B \otimes \F_{q^p}) = 0$, so the inertial types that we are counting are in fact parametrized by elements $m_1 \in (B \otimes \F_{q^p}) \setminus \{0\} \simeq \F_{q^p}^\times$, of which there are~$q^p-1$.
\end{proof}

\renewcommand{\addcontentsline}[3]{}
\emergencystretch=1em
\bibliographystyle{alphaurl}
\bibliography{refs.bib}

\end{document}